\documentclass[11pt]{article}

\def\UseSection{
        \numberwithin{equation}{section}
        \newtheorem{theorem}    {Theorem}[section]
        \DefineTheorems 
}
\usepackage[textwidth=500pt,textheight=650pt,centering]{geometry} 

\usepackage{amsfonts}
\usepackage{amsmath,amssymb,amsthm}
\usepackage{appendix}
\usepackage{bbm} 
\usepackage{amsbsy}
\usepackage{enumerate}
\usepackage{cite}

\usepackage[bookmarks, colorlinks, breaklinks,
pdftitle={},pdfauthor={}]{hyperref}

\hypersetup{
  linkcolor=black,
  citecolor=black,
  filecolor=black,
  urlcolor=black}

\usepackage{verbatim}
\usepackage{tikz}
\usepackage{color}

\newcommand{\black}{\black}

\usepackage[font=small,labelfont=bf]{caption}

\numberwithin{equation}{section}

\usetikzlibrary{calc,decorations.markings}
\usetikzlibrary{decorations.pathmorphing}
\usetikzlibrary{decorations.pathreplacing}
\usetikzlibrary{patterns}
\usetikzlibrary{arrows}

\newcommand{\bb}[1]{\mathbb{#1}}

\newcommand{\1}{\mathbbm{1}}

\newcommand{\blank}[1]{}

\newcommand{\E}{\bb E}
\newcommand{\R}{\bb R}
\newcommand{\Z}{\bb Z}
\newcommand{\C}{\bb C}

\newcommand{\N}{\bb N}
\newcommand{\T}{\bb T}
\renewcommand{\P}{\bb P}

\newcommand{\Wcal}   {\mathcal{W}}

\newcommand{\nnb}	{\nonumber \\}
\newcommand{\bubble}{{\sf B}}
\newcommand{\degree}{\Omega}

\def\DefineTheorems{
	\newtheorem{lemma}      [theorem] {Lemma}
	\newtheorem{cor}        [theorem] {Corollary}
	
	\newtheorem{prop}        [theorem] {Proposition}
	\theoremstyle{definition}
	
	\newtheorem{rk}       [theorem] {Remark}

}

\UseSection
\newcommand{\lbeq}[1]  {\label{e:#1}}
\newcommand{\refeq}[1] {\eqref{e:#1}}    

\title{The near-critical two-point function and the torus plateau\\
for weakly self-avoiding walk in high dimensions}
\author{Gordon Slade\thanks{Department of Mathematics,
     University of British Columbia,
     Vancouver, BC, Canada V6T 1Z2.
     {\tt slade@math.ubc.ca}.
     {\tt https://orcid.org/0000-0001-9389-9497}.
     }}

\begin{document}

\date{\vspace{-5ex}} 

\maketitle

\begin{abstract}
We use the lace expansion to study the long-distance decay of the
two-point function of weakly self-avoiding walk on the integer lattice $\mathbb{Z}^d$ in
dimensions $d>4$, in the vicinity of the critical point, and
prove an upper bound
$|x|^{-(d-2)}\exp[-c|x|/\xi]$, where the correlation length
$\xi$ has a square root divergence at the critical point.
As an application, we prove that the two-point function for weakly self-avoiding
walk on a discrete torus in dimensions $d>4$ has a ``plateau.''
We also discuss the significance and consequences of the plateau for the analysis
of critical behaviour on the torus.
\end{abstract}

%

\section{Introduction and main result}

\subsection{Introduction}

A guiding but generally unproven
principle in the scaling theory for critical phenomena in statistical mechanical models
on $\Z^d$ is that the two-point function
near a critical point generically  has decay of the form
\begin{equation}
\lbeq{Gscaling}
    G_z(x) \approx \frac{1}{|x|^{d-2+\eta}} g(|x|/\xi(z))
\end{equation}
in some reasonable meaning for ``$\approx$'', when $|x|$ is comparable to
the correlation length $\xi(z)$ and $z$ is close
to its critical value $z_c$.
The parameter $z$ depends on the model
and represents, e.g.,
the fugacity for self-avoiding walk, the bond density for bond percolation, or the
inverse temperature for the Ising model.
The universal critical exponent $\eta$ depends on dimension, the correlation length
$\xi(z) \approx (1-z/z_c)^{-\nu}$
diverges as $z \to z_c$  with
a  dimension-dependent universal critical exponent $\nu$,
and $g$ is a function with rapid decay.
The relation \refeq{Gscaling} is a basis for the derivation of the scaling relations
between critical exponents, such as Fisher's relation $\gamma=(2-\eta)\nu$, which are
of primary importance in critical phenomena.  Such derivations can be found for spin
systems in \cite[Section~4.1]{ID89a} (in a physics style), or (in a more mathematical
style) in \cite[Section~9.2]{Grim99} for percolation and in \cite[Section~2.1]{MS93}
for self-avoiding walk.
The relation \refeq{Gscaling} is fundamental in the physics literature but a mathematical justification
is lacking in most examples.

In our main result, Theorem~\ref{thm:mr},
we prove an upper bound of the form \refeq{Gscaling}
for weakly self-avoiding walk in dimensions $d>4$ (for which $\eta=0$), i.e.,
\begin{equation}
\lbeq{Gbd1}
    G_z(x) \le c_0 \frac{1}{1\vee |x|^{d-2}}e^{-c_1m(z) |x|}
    \qquad
    (x \in \Z^d, \; z \in [0,z_c]),
\end{equation}
with $c_1\in (0,1)$ and
with $m(z) = \xi(z)^{-1}$ asymptotic to a multiple of $(1-z/z_c)^{1/2}$ as $z\to z_c$.
We generally write formulas in
terms of the \emph{mass} $m(z)$ rather than the correlation length $\xi(z)=m(z)^{-1}$.
The norm $|x|$ denotes the Euclidean norm $\|x\|_2$,
and $a\vee b= \max\{a,b\}$ is used in \refeq{Gbd1} to avoid division by zero.

The inequality \refeq{Gbd1} is a one-sided inequality so is not as strong as the
assertion \refeq{Gscaling}.  Nevertheless, the upper bound is arguably more useful
than a lower bound, and it has remarkable new consequences.
One application of \refeq{Gbd1} is given in Theorem~\ref{thm:plateau}, where we prove
that the decay of the two-point function
for weakly self-avoiding walk on a discrete torus in dimensions $d>4$ has a ``plateau,''
in the sense that it decays like the $\Z^d$ two-point function for small $x$ and then
ceases to decay and levels off for larger $x$ in the torus.
As discussed in more detail in Section~\ref{sec:torus},
the effect of periodic vs free boundary conditions on a two-point function above the
upper critical dimension has been
the subject of some debate in the physics literature.
It is therefore useful to have theorems which
provide definitive statements concerning the plateau.
Furthermore, \refeq{Gbd1} and the existence of the plateau
are used in \cite{MS22} to provide a partial answer to the
question: how long does a weakly self-avoiding walk on a discrete torus in dimensions $d>4$ have
to be before its behaviour differs significantly
from that of a weakly self-avoiding walk on $\Z^d$?
(The answer is at least length $r^{d/2}$ on a torus of volume $r^d$.)
See also \cite{Mich22} for related work concerning the Brownian scaling limit on the torus.

For spread-out
percolation in dimensions
$d>6$ an analogue of \refeq{Gbd1} is proved in \cite{HMS22} by methods
very  different from the
ones we use here,
and this allows the existence of the torus plateau to be proved via an adaptation of
the proof of Theorem~\ref{thm:plateau}.  Among other consequences, this leads to a proof
of the percolation triangle condition for a high-dimensional
torus (as defined in \cite{BCHSS05a}) \emph{without} the
need for any torus lace expansion as in \cite{BCHSS05b}---the torus critical behaviour
can be inferred directly from consequences obtained via the percolation lace expansion on $\Z^d$, which provides a new perspective on the much studied subject of torus percolation.
These two applications of the torus plateau, to weakly self-avoiding walk
in \cite{MS22} and to percolation in \cite{HMS22}, indicate that it is
an invaluable tool for the analysis of torus critical phenomena.

As a byproduct and as an instructive comparison,
we also give a short proof of the existence of a plateau for
the torus two-point function for simple random walk in dimensions $d>2$;
this plateau was announced in \cite{ZGFDG18} and proved
in \cite{Grim18,ZGDG20} using methods different from ours.

\subsection{The model}
\label{sec:model}

Our results are for the nearest-neighbour weakly self-avoiding walk in dimensions $d>4$.
Background material can be found in \cite{MS93,Slad06}.  The model is defined as follows.

For $d\in\N$,
let $D:\Z^d \to \R$ be the one-step transition probability for simple random walk on $\Z^d$,
i.e.,
$D(x)=\frac{1}{2d}$ if $|x|=1$ and otherwise $D(x)=0$.
For $n \in \N$, let $D^{*n}$ denote the $n$-fold
convolution of $D$ with itself; this is the $n$-step transition probability.
We adopt the convention that $D^{*0}(x) = \delta_{0,x}$.
Let $\Wcal_n(x)$ denote the set of $n$-step walks from $0$ to $x$, i.e., the
set of $\omega = (\omega(0),\omega(1),\ldots,\omega(n))$ with each $\omega(i)\in \Z^d$,
$\omega(0)=0$, $\omega(n)=x$,
and $|\omega(i)-\omega(i-1)|=1$ for $1 \le i \le n$.
The set $\Wcal_0(x)$ consists of the zero-step walk $\omega(0)=0$ when $x=0$,
and otherwise it is the empty set.  We write $\degree = 2d$
for the degree of the nearest-neighbour graph.
The \emph{simple random walk two-point function} (also called the
\emph{lattice Green function}) is
defined, for $z \in [0,\frac{1}{\degree}]$, by
\begin{equation}
\lbeq{Czdef}
    C_z(x) = \sum_{n=0}^\infty \sum_{\omega\in \Wcal_n(x)} z^n
    =
    \sum_{n=0}^\infty (z\degree)^n D^{*n}(x)
    \qquad
    (x \in \Z^d).
\end{equation}

For an $n$-step walk $\omega$, and for $0 \le s < t \le n$, we define
\begin{equation}
    U_{st}(\omega) =
    \begin{cases}
        -1 & (\omega(s)=\omega(t))
        \\
        0 & (\omega(s)\neq \omega(t)).
    \end{cases}
\end{equation}
Given $\beta \in (0,1)$, $z \ge 0$, and $x \in \Z^d$,
the \emph{weakly self-avoiding walk two-point function} is then
defined by
\begin{equation}
\lbeq{Gzdef}
    G_z(x) = \sum_{n=0}^\infty \sum_{\omega\in \Wcal_n(x)} z^n \prod_{0\le s<t\le n}(1+\beta U_{st}(\omega)).
\end{equation}
Compared to \refeq{Czdef}, the product in \refeq{Gzdef} discounts each $\omega$ by a factor
$1-\beta$ for each pair $s,t$ with an intersection for $\omega$, hence the name
``weakly self-avoiding walk.''  The choice $\beta=1$ defines the \emph{strictly} self-avoiding
walk.
The \emph{susceptibility} is defined by
\begin{equation}
    \chi(z)=\sum_{x\in \Z^d} G_z(x).
\end{equation}
A standard subadditivity argument implies the existence of $z_c=z_c(\beta) \ge
z_c(0)=\frac{1}{\degree}$ such that
$\chi(z)$ is finite if and only if $z \in [0,z_c)$; also $\chi(z) \ge 1/(1-z/z_c)$
so $\chi(z_c)=\infty$ (see, e.g., \cite[Theorem~2.3]{Slad06}).
In particular, this implies that the series \refeq{Gzdef} converges at least for $z \in [0,z_c)$.

\subsection{Main result}

Our main result is the following theorem.  Its proof, which uses the
Brydges--Spencer lace
expansion \cite{BS85}, is inspired by
the methods of \cite{Slad20_lace} and
\cite[Section~6.5.1]{MS93} (the latter is based on \cite{HS92a}) but is more than
the union of these methods.
Its statement involves the mass (inverse correlation length)
$m(z)$; this is the exponential decay rate of the subcritical two-point function
and is defined explicitly in \refeq{mzdef}.

We write $f \sim g$ to mean $\lim f/g =1$, and
$f \asymp g$ to mean that $c_1 f \le g \le c_2 f$
with $c_1,c_2>0$.  In general, constants are permitted to depend on the dimension $d$, which
is a fixed parameter in our analysis.

\begin{theorem}
\label{thm:mr}
Let $d>4$ and let $\beta$ be sufficiently small.
There are constants $c_0>0$ and $c_1\in (0,1)$, which depend on $d$ but not on $\beta$,
such that for all $z\in (0, z_c)$ and $x\in\Z^d$,
\begin{equation}
\lbeq{Gmr}
    G_z(x) \le c_0 \frac{1}{1\vee |x|^{d-2}} e^{-c_1 m(z)|x|}.
\end{equation}
The mass has the asymptotic form $m(z) \sim c(1-z/z_c)^{1/2}$ as $z \to z_c$,
with constant $c=\degree^{1/2} + O(\beta)$.
\end{theorem}

In the proof, the order of operations is:
\begin{enumerate}
\item
Prove that (with constant independent of $\beta$)
\begin{equation}
\lbeq{Gzcub}
    G_{z_c}(x) \le c_0 \frac{1}{1\vee |x|^{d-2}}.
\end{equation}
This has been proved already in \cite{Slad20_lace} and also in \cite{BHK18}, in fact as
an asymptotic relation rather than only as a bound, and the
proof is not repeated here.
Those proofs use the lace expansion with a bootstrap argument varying $z$.
\item
Prove that $m(z)  \sim {\rm const}(1-z/z_c)^{1/2}$ as $z \to z_c$.
An important element of the proof is the control of the lace
expansion ``tilted'' by $e^{m(z)x_1}$, where $x_1$ is the first component of $x \in \Z^d$.
The proof uses a bootstrap argument varying $m$, as in \cite{Hara90,MS93,HS92a}.
\item
Prove the inequality \refeq{Gmr}.   The proof
is based on the method of \cite{Slad20_lace} but now applied to the exponentially
tilted two-point function, which requires a significant extension of \cite{Slad20_lace}.
\end{enumerate}

The proof has the potential to apply
to spread-out strictly self-avoiding walk, to the Ising model, and to the $\varphi^4$ model,
above the upper critical dimension $4$.  However this
would need further development.
It does not apply in its present form to percolation, nor to lattice trees and lattice
animals; see Remark~\ref{rk:perc}.  A different approach is used to prove the
analogue of \refeq{Gmr} for percolation in high dimensions in
\cite{HMS22}, using methods special to percolation which do not apply to weakly self-avoiding
walk.

\subsection{Decay of the two-point function}

To place the estimate \refeq{Gmr} in context, we summarise what was previously known
about the decay of the two-point function.
The theory for the two-point function is more developed for
the more difficult case of strictly self-avoiding
walk ($\beta=1$) than for weakly self-avoiding walk, but typically the adaptation of
proofs from the former to the latter is straightforward since the proofs are based
on subaddivitity arguments that apply equally well to both cases.
\begin{enumerate}
\item
Let $z\in (0,z_c)$ and $d \ge 2$.
There is a $z$-dependent norm $|\cdot|_z$ on $\R^d$, with $\|x\|_\infty \le |x|_z
\le \|x\|_1$ for all $x \in \R^d$, such that the \emph{mass} $m(z)$
defined by the limit
\begin{equation}
\lbeq{mzdef}
    m(z) = \lim_{|x|_z \to\infty} \frac{-\log G_z(x)}{|x|_z}
\end{equation}
exists in $(0,\infty)$.  The function $m$  is continuous
and strictly decreasing in $z$,
$m(z) \to \infty$ as $z \to 0$,
and $m(z) \to 0$ as $z \to z_c$.
These facts are proved in  \cite[Section~4.1]{MS93} for strictly self-avoiding walk;
the proofs directly adapt to the weakly self-avoiding walk.
\item
Let $z\in (0,z_c)$ and $d \ge 2$.
The \emph{bubble diagram} is defined by $\bubble(z) = \sum_{x\in\Z^d}G_z(x)^2$.
The two-point function satisfies the inequality
\begin{equation}
\lbeq{Gbubblebd}
    G_z(x) \le \bubble(z)^{1/2} e^{-m(z)|x|_z} \qquad (x \in \Z^d).
\end{equation}
This is proved in \cite[Theorem~4.1.18]{MS93} for the strictly self-avoiding walk,
and the proof adapts directly to the weakly self-avoiding walk.
\item
Let $z\in (0,z_c)$ and $d \ge 2$.
The two-point function obeys
the Ornstein--Zernike decay
\begin{equation}
\lbeq{OZ}
    G_z((x_1,0,\ldots,0)) \sim c_z \frac{1}{x_1^{(d-1)/2}} e^{-m(z)x_1}
    \qquad (x_1\to \infty),
\end{equation}
with $c_z>0$.  Off-axis behaviour is also known.
This is proved in \cite{CC86b}
and \cite[Theorem~4.4.7]{MS93} for the strictly self-avoiding
walk.
The bound \refeq{OZ} exhibits a power-law correction $x_1^{-(d-1)/2}$ to the exponential decay, which is different from the power $|x|^{-(d-2)}$ in \refeq{Gmr} or \refeq{Gzcub}
(unless $d=3$).
\item
Let $z\in (0,z_c)$ and $d>4$.
The asymptotic behaviour of the mass is
\begin{equation}
\lbeq{masy}
    m(z) \sim {\rm const}\, (1-z/z_c)^{1/2}  \qquad   (z \to z_c).
\end{equation}
This is a statement that the critical exponent $\nu$ is equal to $\frac 12$.
For the strictly self-avoiding walk,
\refeq{masy} is proved in \cite{HS92a} and
\cite[Theorem~6.1.2]{MS93}  using the lace expansion.
In Section~\ref{sec:masy}, we indicate the small changes need to
prove \refeq{masy} for weakly self-avoiding walk.
The method of proof was first developed for percolation \cite{Hara90}, and its elementary version for simple random walk is given in \cite[Theorem~A.2]{MS93}.
\item
Let $d>4$.
The critical two-point function has asymptotic behaviour
\begin{equation}
\lbeq{Gzc}
    G_{z_c}(x)  \sim {\rm const}\, \frac{1}{|x|^{d-2}} \qquad (|x|\to\infty).
\end{equation}
This is proved for strictly self-avoiding walk in \cite{HHS03,Hara08},
for weakly self-avoiding walk in \cite{BHK18,Slad20_lace}, and for the continuous-time
weakly self-avoiding walk (also known as the lattice Edwards model) in \cite{BHH19}.
All these proofs use the lace expansion.
The formula \refeq{Gzc} is a statement that the critical exponent $\eta$ is equal to
zero\footnote{For the more difficult dimension $d=4$, $\eta=0$ is proved  for the continuous-time
weakly self-avoiding walk in
\cite{BBS-saw4}.  For dimensions
$d=2,3$ it is predicted but not proved that the decay is of the form $|x|^{-(d-2+\eta)}$
with $\eta =\frac{5}{24}$ for $d=2$ and $\eta \approx 0.03$ for $d=3$.}
for $d>4$.
\end{enumerate}

A consequence of \refeq{Gzc}
is that the critical bubble diagram $\bubble(z_c)$ is finite, which
in turn implies that the susceptibility obeys $\chi(z) \asymp (1-z/z_c)^{-1}$
(see, e.g., \cite[Theorem~2.3]{Slad06}).
This is a statement that the critical exponent $\gamma$ is equal to $1$.
A stronger asymptotic formula for the susceptibility
has also been proved---we recall in \refeq{chiasy} the proof that for $d>4$ and small $\beta$
\begin{equation}
\lbeq{chiasy0}
    \chi(z) \sim A (1-z/z_c)^{-1}  \qquad   (z \to z_c)
\end{equation}
with $A=1+O(\beta)$.
The same is proved for the strictly self-avoiding walk in \cite{HS92a}.

The importance of the estimate \refeq{Gmr} resides in its uniformity as $z \to z_c$.
Indeed, if we consider only $z \in (0,z_c-\delta]$ with fixed $\delta>0$,
then by \refeq{Gbubblebd} and the fact that $\bubble(z_c-\delta)<\infty$,
for any $c_1\in (0,1)$ we have
\begin{align}
\lbeq{smallz1}
    G_z(x) &
    \le
    {\rm const}\, e^{-(1-c_1)m(z)\|x\|_\infty} e^{-c_1m(z)\|x\|_\infty}.
\end{align}
Since $m(z) \ge m(z_c-\delta)$, and since $|x|\le d^{1/2}\|x\|_\infty$,
\begin{equation}
\lbeq{smallz2}
    e^{-(1-c_1)m(z)\|x\|_\infty} \le
    e^{-(1-c_1)m(z_c-\delta)\|x\|_\infty}
    \le {\rm const} \frac{1}{|x|^{d-2}}.
\end{equation}
Thus the bound \refeq{Gmr} holds for $z\in (0,z_c-\delta]$, so our focus
needs to be on $z$ close to $z_c$.

\subsection{Conjectured decay}

The precise asymptotic behaviour of the lattice Green function $C_z(x)$ for fixed $z$ less
than the critical value $z_0= \frac{1}{\Omega}$ has recently been elucidated
in \cite{MS21}.  It is proved there that there is an explicit norm $|\cdot|_z$ on $\R^d$
(not the same as the norm in \refeq{mzdef})
that interpolates monotonically between the limiting values $\lim_{z\to 0}|x|_z = \|x\|_1$
and $\lim_{z \to z_0}|x|_z = \|x\|_2$, and an explicitly defined function $m_0(z)$
with $m_0(z) \sim {\rm const} \, (1-z/z_0)^{1/2}$ as $z \to z_0$
(see \refeq{m0def}--\refeq{m0asy}), such that, as $n \to \infty$,
\begin{equation}
\lbeq{LGasy}
    C_z(nx)
    \sim
    c_{z,\hat x} \frac{1}{(n|x|_z)^{d-2}} (m_0(z)n|x|_z)^{(d-3)/2} e^{-  m_0(z)n|x|_z}
    \qquad
    (\text{nonzero} \; x \in \Z^d).
\end{equation}
The constant $c_{z,\hat x}$ has a limit as $z \to z_0$ which is positive and independent of
the direction $\hat x = \frac{x}{|x|_z}$.
This is consistent with \refeq{Gscaling} with $g(t) = t^{(d-3)/2} e^{-t}$.
It is natural to conjecture that \refeq{LGasy} will also apply to statistical mechanical
models above their upper critical dimensions, including self-avoiding walk and the Ising model
for $d>4$, and percolation for $d>6$.
A potential factor such as $(m(z)|x|)^{(d-3)/2}$ is
compensated by giving up some exponential decay
in the factor $e^{-c_1m|x|}$ in \refeq{Gmr}, with constant $c_1<1$.

\subsection{The plateau for the torus two-point function}
\label{sec:torus}

Let $\T_r^d = (\Z/r\Z)^d$ denote the discrete $d$-dimensional torus of period $r\ge 3$.
We are interested in large $r$, and in obtaining estimates that remain valid uniformly in
large $r$.
For notational
convenience, we sometimes evaluate a $\Z^d$ two-point function at a point $x\in\T_r^d$,
with the understanding that in this case we identify
$x$ with a point in $[-r/2,r/2)^d \cap \Z^d$.

Let $G^{\T}_z(x)$ denote the analogue of \refeq{Gzdef}
for walks $\omega$
on the torus rather than on $\Z^d$, and let
\begin{equation}
    \chi^\T(z) = \sum_{x\in\T_r^d}G^{\T}_z(x)
\end{equation}
denote the torus susceptibility.
For any $n$-step torus walk $\omega$, by the pigeon-hole principle
there must be a vertex that is visited by $\omega$ at least $N_n=\lceil nr^{-d}\rceil$ times, so
\begin{equation}
    \prod_{0\le s<t\le n}(1+\beta U_{st}(\omega)) \le (1-\beta)^{\binom{N_n}{2}}.
\end{equation}
For fixed $r$, the binomial coefficient
$\binom{N_n}{2}$  grows quadratically in $n$,
so since the number of $n$-step walks is at most $\degree^n$, for any $z \ge 0$ we have
\begin{equation}
    \chi^\T(z) \le \sum_{n=0}^\infty (z\degree)^n (1-\beta)^{\binom{N_n}{2}} < \infty.
\end{equation}
The susceptibility and two-point function on the torus are therefore entire functions of $z$
when $\beta >0$.
This is in contrast to the situation for $\Z^d$, where convergence is limited to the
disk of radius $z_c$.

For $r \ge 3$,
walks on the torus are in one-to-one correspondence with walks on $\Z^d$ via
the canonical projection from $\Z^d$ to $\T_r^d$.  We refer to the $\Z^d$ walk
corresponding to a torus walk as the ``unfolding'' of the torus walk.\footnote{In detail,
if $\omega$ is a torus walk starting at $0$
then its unfolding $\bar\omega$ is the walk on $\Z^d$ with
$\bar\omega(0)=0$ and $\bar\omega(i) = \bar\omega(i-1)+ (\omega(i)-\omega(i-1))_r$ where
for $x \in \T_r^d$ we write $x_r$ for its unique representative in $[-\frac r2,\frac r2)
\cap \Z^d$.}  Since the unfolding of a  walk cannot have more self-intersections than
the walk itself, for all dimensions, for all $\beta\in [0,1]$, and for all $z\in [0,z_c]$ we have
\begin{equation}
\lbeq{chiTlechi}
    \chi^\T(z) \le \chi(z).
\end{equation}

The proofs of the theorems discussed in this section, Theorems~\ref{thm:plateau}--\ref{thm:plateau-srw}, are given
in Section~\ref{sec:plateaupf}.

\subsubsection{The plateau for weakly self-avoiding walk for $d>4$}

The following theorem is one application of Theorem~\ref{thm:mr}.
It shows that the slightly subcritical weakly self-avoiding walk torus two-point function
has asymptotic form given by two terms:  an $x$-dependent term that is simply
the $\Z^d$ two-point function, plus a constant term that can dominate for large $x$.
The constant term is the torus ``plateau.''

\begin{theorem}
\label{thm:plateau}
Let $d>4$ and let $\beta$ be sufficiently small.
There are constants $c_i>0$, depending on $d$ but not on $\beta$,
such that for all $x \in \T_r^d$,
\begin{equation}
\lbeq{plateau}
    G_{z}(x) +c_1 \frac{\chi(z)}{r^d}
    \le
    G^{\T}_{z}(x)
    \le
    G_{z}(x) +c_2 \frac{\chi(z)}{r^d},
\end{equation}
where the upper bound holds for all $r \ge 3$ and all $z \in (0,z_c)$, whereas
the lower bound holds provided that
$z \in [ z_c-c_3r^{-2},   z_c-c_4\beta^{1/2}r^{-d/2}]$.
\end{theorem}

Note that the interval of $z$ values for the lower bound is nonempty for sufficiently
large $r$ since $d>4$.
Of course the upper bound of \refeq{plateau} also holds for $z=z_c$
since $\chi(z_c)=\infty$.  The lower bound cannot hold when $z=z_c$ for the same reason,
because, as noted above, $G_z^{\T}(x)$ is an entire function
of $z$ and hence is finite for \emph{all} $z \ge 0$.
By universality we expect that \refeq{plateau}
also holds for the strictly self-avoiding walk in dimensions $d>4$; Monte Carlo verification
of this has been carried out in \cite{ZGFDG18,ZGDG20}.

Let $\rho=z_c-z$.
In the proof of Theorem~\ref{thm:plateau} (see Remark~\ref{rk:Gamma}), we show  that
(small) $c_3$ can be chosen such that
\begin{equation}
\lbeq{Gzzc}
    G_z(x) \asymp \frac{1}{1\vee |x|^{d-2}}
    \quad
    \text{when $\rho \in [0,c_3 r^{-2}]$ and $x\in\T_r^d$.}
\end{equation}
It therefore follows from \refeq{plateau} that
\begin{equation}
\lbeq{plateaux}
    G^{\T}_{z}(x)
    \asymp
    \frac{1}{1\vee |x|^{d-2}} +  \frac{\chi(z)}{r^d}
    \quad \text{when $\rho \in [ c_4\beta^{1/2}r^{-d/2},c_3r^{-2}]$ and $x\in\T_r^d$.}
\end{equation}
By \refeq{chiasy0},
the susceptibility diverges linearly at the critical point,
so $\chi(z) \asymp r^{p}$ when $\rho \asymp r^{-p}$, and hence
\begin{equation}
\lbeq{plateaup}
    G^{\T}_{z}(x)
    \asymp
    \frac{1}{1\vee |x|^{d-2}} +  \frac{1}{r^{d-p}}
    \quad\text{when $\rho \asymp \frac{1}{r^p}$  with $p\in [2,\tfrac d2]$
     and $x\in\T_r^d$.}
\end{equation}
For $p>2$ the constant term dominates when $|x|^{d-2} \ge r^{d-p}$,
i.e., when $|x| \ge r^{(d-p)/(d-2)}$.  The latter domain of $x$ is the ``plateau''
where the torus two-point function no longer decays with distance.

\subsubsection{The scaling window}

In models defined on a finite graph such as a torus, rather than on an infinite graph
such as $\Z^d$, there is not a precise notion of a critical point.   Unlike the
susceptibility $\chi$ which diverges at $z_c$,
the torus susceptibility $\chi^\T$ is an entire function with no singularity.
Instead, there is a notion of a \emph{scaling window} (or \emph{critical window}),
which is an interval of parameter values where the model exhibits critical behaviour.

Self-avoiding walk on the complete graph is exactly solvable \cite{Slad20,DGGNZ19}.
Consistent with the discussion in \cite{Slad20} for the complete graph, it is natural
to define the scaling window for self-avoiding walk on a torus
in dimensions $d>4$ to consist
of the values of $z$ with $|z-z_c|\le O(r^{-d/2})$.  Motivated by
the complete graph, the susceptibility is conjectured to obey $\chi^\T(z) \asymp
r^{d/2}$ when $z$ is in the scaling window.\footnote{After this work was completed,
it has been conjectured in \cite{MPS23} that
for $d>4$ and all $\beta\in (0,1]$
 the universal profile for the
susceptibility in the scaling window
is given by $\lim_{r \to \infty} r^{-d/2}\chi^{\T}(z_c(1+\lambda_1 s r^{-d/2}))
= \lambda_2 \int_0^\infty t\, e^{-\frac 14 t^4 - \frac 12 st} dt$ for all $s \in \R$,
for some constants $\lambda_1,\lambda_2$ depending on $d$ and $\beta$ (but not on $s,r$). }
For weakly self-avoiding walk,
constants in these two relations naturally depend on $\beta$, since $\beta=0$ is simple
random walk for which there is no scaling window because even on the torus the susceptibility
diverges at $z=\frac{1}{\degree}$.
The largest value $z=z_c-c_4\beta^{1/2}r^{-d/2}$ that is permitted for the lower bound in
Theorem~\ref{thm:plateau}
lies in this scaling window.  For notational convenience, we write this value
as $z_*$.
Summation of the lower bound of \refeq{plateau} over $x \in \T_r^d$,
together with the general upper bound of \refeq{chiTlechi},
leads to the following immediate corollary
of Theorem~\ref{thm:plateau}.
In particular, with
\refeq{chiasy0}, the corollary shows that $\chi^\T (z_*) \asymp r^{d/2}$, consistent with the above
mentioned conjecture.

\begin{cor}
The two susceptibilities obey the inequality
\begin{equation}
    c_1 \chi(z) \le \chi^\T(z) \le  \chi(z),
\end{equation}
with the upper bound valid for general $d,\beta,z$, and the lower bound valid
for $d>4$, sufficiently small $\beta>0$, and under the same restrictions on $z$
as in Theorem~\ref{thm:plateau}.
\end{cor}

As a comparison, we note that
the scaling window for high-dimensional percolation on a torus has been extensively studied;
a summary can be found in \cite[Chapter~13]{HH17book}.
For percolation,
the torus scaling window consists of points $p$ with $|p-p_c| \le O(r^{-d/3})$.
In particular, the $\Z^d$ critical point $p_c$ lies in the scaling window.
In the scaling window, the torus susceptibility is of order $r^{d/3}$.
 Note the different exponent $r^{d/3}$ appearing
here, compared with $r^{d/2}$ for self-avoiding walk.
A new perspective on high-dimensional torus percolation is provided in \cite{HMS22},
based on a percolation version of Theorem~\ref{thm:mr} (proved via different means
than those used here) and of Theorem~\ref{thm:plateau} (proved using similar means
to those used here).  An advantage of this new perspective is that it obviates the need for the
torus lace expansion \cite{BCHSS05b}.  Instead, critical torus percolation can be analysed
using only results derived from the $\Z^d$ lace expansion, particularly the $|x|^{-(d-2)}$
decay of the two-point function \cite{HHS03,Hara08}, combined with percolation analogues
of Theorems~\ref{thm:mr}--\ref{thm:plateau}.

\subsubsection{Boundary conditions and the plateau}

There is a sizeable literature about the effect of boundary conditions
on the decay of the two-point function of
statistical mechanical models on a finite box above the upper critical
dimension, including \cite{ZGFDG18,WY14,LM16}.
This literature has included
some debate about what should be the correct behaviour.
The emerging consensus is that with free boundary conditions there is no
plateau and the susceptibility behaves as $r^2$ at the $\Z^d$ critical point
(for Ising or self-avoiding walk)
whereas with periodic boundary conditions (the torus)
there is a plateau and the susceptibility
behaves instead as $r^{d/2}$.

For the Ising model in dimensions $d>4$, with free boundary conditions
the lack of a plateau and the $r^2$ behaviour
of the susceptibility has been proved
in \cite{CJN21}.  With periodic boundary conditions, a plateau lower bound  is
proved in \cite{Papa06}, as is an $r^{d/2}$ lower bound for the susceptibility.
The upper bounds remain unproved on the torus.

The results for percolation are consistent with this picture, again with
the susceptibility of order $r^2$
for free boundary conditions but now of order $r^{d/3}$ on the torus.
For free boundary conditions,
the lack of a plateau at the $\Z^d$ critical point $p_c$ is proved in \cite{CH20}.
For periodic boundary conditions, a plateau upper bound at $p_c$ is proved in
\cite[Theorem~1.7]{HS14}, and plateau upper and lower bounds are proved throughout
the entire scaling window in \cite{HMS22}.

Theorem~\ref{thm:plateau} provides a contribution to this general discussion by
proving existence of a plateau for weakly self-avoiding walk in dimensions $d>4$,
up to the scaling window.

\subsubsection{The plateau for simple random walk for $d>2$}
\label{sec:plateau-srw}

A more elementary version of Theorem~\ref{thm:plateau}
for simple random walk is given in the following theorem, which is a byproduct of our
proof of Theorem~\ref{thm:plateau}.
Let $C^{\T}_z(x)$ denote the analogue of \refeq{Czdef} for walks $\omega$
on the torus rather than on $\Z^d$.  Let $z_0 = \frac{1}{\degree}=\frac{1}{2d}$.
The susceptibility $\chi_0(z)$ is the same for simple random walk on the torus or on $\Z^d$:
\begin{equation}
\lbeq{chi0def}
    \chi_0(z) = \sum_{x\in \Z^d} C_z(x) = \sum_{x\in \T_r^d} C_z^{\T}(x) = \frac{1}{1-z\degree}
    \qquad
    (z \in [0,z_0)).
\end{equation}
The isolation of $d=4$ in the theorem is an unnatural artifact of our proof.

\begin{theorem}
\label{thm:plateau-srw}
Let $d>2$.
For $d \neq 4$, there are constants $c_i'>0$ such that for all $x \in \T_r^d$,
\begin{equation}
\lbeq{plateau-srw}
    C_{z}(x) +c_1' \frac{\chi_0(z)}{r^d}
    \le
    C^{\T}_{z}(x)
    \le
    C_{z}(x) +c_2' \frac{\chi_0(z)}{r^d},
\end{equation}
where the upper bound holds for all $r \ge 3$ and all $z \in (0,z_0)$, whereas
the lower bound holds provided that
$z\in [z_0- c_3'r^{-2},z_0)$.
For $d=4$ the upper bound also holds as stated, but the constant term in the lower bound is weakened to
$c_3' \frac{\chi_0(z)}{r^d}\frac{1}{\log \chi_0(z)}$ with restriction
that $\rho|\log\rho|r^2$
is sufficiently small, where $\rho = z_0-z$.
\end{theorem}

Note that both inequalities in \refeq{plateau-srw} hold as equalities when $z=z_0$,
since $\chi_0(z_0)=\infty$ by \refeq{chi0def}, and $C^{\T}_{z_0}(x)=\infty$
because simple random walk on the torus is recurrent.
We indicate in Remark~\ref{rk:Gamma} that $C_z(x)$ obeys \refeq{Gzzc}
(now with $\rho=z_0-z$),
and hence \refeq{plateaup}
also holds for $C_z^{\T}(x)$, now for all $p \ge 2$.
Thus Theorem~\ref{thm:plateau-srw} implies
that the  ``plateau'' concept also applies to simple random walk
(with an
unnatural logarithmic caveat
for $d=4$).

A different and earlier proof of Theorem~\ref{thm:plateau-srw} is
based on the local central limit
theorem \cite{Grim18,ZGFDG18,ZGDG20}.
In fact, the results of \cite{Grim18,ZGFDG18,ZGDG20} are more general and for $d=4$ are stronger since they
do not have the logarithmic correction in the lower bound.
Our proof instead uses the heat kernel estimate stated in \refeq{hk}
below (this has analogues for very general random walks), which
affords a relatively simple proof of \refeq{plateau-srw}.
As our principal interest is the weakly self-avoiding walk in dimensions $d>4$, we do not aim
for generality in Theorem~\ref{thm:plateau-srw}, nor attempt to eliminate the logarithm for $d=4$.

Theorem~\ref{thm:plateau-srw} has the following heuristic interpretation which provides
some insight into the origin of the plateau.
Consider the nearest-neighbour
torus random walk $S_n^{\T}$ subjected to $z$-dependent killing, i.e.,
the walk has length $N$
with geometric probability $\P(N=n)=(z\degree)^n(1-z\degree)$ for $n \ge 0$, with the
random variable $N$ independent of the walk's steps.
The two-point function $C^{\T}_{z}(x)$ is the expected number of visits to $x$
by the torus walk subjected to $z$-dependent killing:
\begin{equation}
\lbeq{CN}
    C_z^{\T}(x) = \E \Big(\sum_{n=0}^N \1_{S_n^{\T} =x}\Big)
    \qquad (x\in \T_r^d),
\end{equation}
as can be verified by computing the right-hand side via conditioning on $N$.
The expected length is $\E N = z\degree \chi_0(z)$, so
the susceptibility $\chi_0(z)=(1-z\degree)^{-1}$ is simply $1+\E N$.
As discussed above \refeq{chiTlechi}, for $r \ge 3$
walks on the torus are in one-to-one correspondence with walks on $\Z^d$ via
unfolding.
Thus,
a torus walk to $x$ unfolds to a walk on $\Z^d$ ending at $x$ or at a point
$x+ru$ with $u$ a nonzero point in $\Z^d$.
The term $C_{z}(x)$ in \refeq{plateau-srw} is the
expected number of visits to $x$ by torus walks from $0$ which unfold to walks on $\Z^d$
which end at $x$.
In the proof of Theorem~\ref{thm:plateau-srw},
the term $r^{-d}\chi_0$ arises from walks that wrap around the
torus---these unfold to walks on $\Z^d$ that end at $x+ru$
for some nonzero $u \in \Z^d$.  For $z$
with $z_0-z \le r^{-2}$, the expected length is at least $r^2$,
so the torus walk is well mixed (see \cite[Theorem~5.5]{LP17})
and its location is close to
uniformly random on the torus.
On average, it therefore spends time $r^{-d}\chi_0$ at each torus
point, resulting in the constant term in \refeq{plateau-srw}.

An alternate heuristic interpretation is the following, which has been pointed out in
\cite{WY14}.
The Fourier dual of the torus $\T_{r}^d$ is
$\hat\T_{r}^{d} = \frac{2\pi}{r}\T_{r}^d$.
Let $k \cdot x = \sum_{j=1}^d k_jx_j$ denote the dot product
of $k \in \hat \T_{r}^{d}$ with $x \in \T_{r}^d$.
The Fourier transform of $f: \T_{r}^d \to {\mathbb C}$
is defined by
$\hat{f}(k) = \sum_{x \in \T_{r}^d}f(x) e^{ik \cdot x}$
for $k \in \hat\T^{d}_{r}$, and
the inverse Fourier transform is $f(x)
    = \frac{1}{r^d} \sum_{k \in \hat\T_{r}^{d}} \hat{f}(k) e^{-ik \cdot x}$
for $x\in \T_{r}^d$.
In particular,
\begin{align}
\lbeq{CTFourier}
    C^{\T}_{z}(x) = \frac{1}{r^d} \sum_{k \in \hat\T_{r}^{d}} \hat{C}^{\T}_{z}(k) e^{-ik \cdot x}
    \qquad
    (x \in \T_r^d).
\end{align}
The Fourier transform $\hat C_z(k)$ on $\Z^d$ (see \refeq{Chatk})
and its torus counterpart $\hat C_z^{\T}(k)$ have
the same functional form $(1-z\degree \hat D(k))^{-1}$---only the domains for $k$ differ.
Thus the $k=0$ term in \refeq{CTFourier}
is $r^{-d}\hat{C}^{\T}_{z}(0) = r^{-d}\chi_0(z)$,
so the constant term in \refeq{plateau-srw} arises as the zero mode.
The sum over nonzero $k$ is a Riemann sum approximation to the Fourier integral over
the continuum torus $(\R / 2 \pi \Z )^d$
that equals the two-point function $C_z(x)$ for $\Z^d$ via inverse Fourier
transformation.  If the
Riemann sum has the same large-$x$ behaviour as the integral, then we would find that
$C^{\T}_{z}(x) - r^{-d}\chi_0(z)$ and $C_z(x)$ are comparable,
as in \refeq{plateau-srw}.
Verification of this last step is an oscillatory integral problem, which may be difficult.
Nevertheless this
heuristic discussion does support the idea that the plateau arises from the zero
mode.

\section{The lattice Green function}

Recall the definition of the lattice Green function $C_\mu(x)$ in \refeq{Czdef},
where now we write $\mu$ instead of $z$.
For the proof of Theorem~\ref{thm:mr},
we need a version of \refeq{Gmr} for $C_\mu(x)$, as well as estimates on the Fourier transform
of $C_\mu(x)$.  We develop these here.

\subsection{Massive decay}

We now prove an inequality of the form \refeq{Gmr} for $C_\mu(x)$.
Let $\mu_c = \frac{1}{\degree}= \frac{1}{2d}$.
For $\mu \in (0,\mu_c]$, we define $m_0(\mu)\ge 0$ to be the unique solution to
\begin{equation}
\lbeq{m0def}
    \cosh m_0(\mu) = 1+ \frac{1-\mu\degree}{2\mu}.
\end{equation}
In particular, $m_0(\mu_c)=0$,
$m_0$ is a strictly positive strictly decreasing function of $\mu \in (0,\mu_c)$, and
\begin{equation}
\lbeq{m0asy}
    m_0(\mu)^2 \sim  \frac{1}{\mu} -\degree \sim \frac{1}{\mu_c} (1- \mu/\mu_c)
    \qquad
    (\mu \to \mu_c).
\end{equation}
This $m_0(\mu)$ is the
exponential rate of decay of $C_\mu(x)$ in \refeq{LGasy}, and
$C_\mu(x) \le C_\mu(0)e^{-m_0(\mu)\|x\|_\infty}$
for all $x \in \Z^d$ and $\mu \in (0,\mu_c)$; both of these statements are
proved in \cite[Theorem~A.2]{MS93}.
The following proposition is a variation of the above exponential estimate,
with a power law correction as in \refeq{Gmr}.  By \refeq{LGasy}, it is impossible
for \refeq{srwbd} to hold with $a_1=1$ (at least for $d \neq 3$).

The proof of the proposition uses the fact there exist $a,A>0$ such that
the $n$-step transition probability obeys
\begin{equation}
\lbeq{hk}
    D^{*n}(x)
    \le A\frac{1}{n^{d/2}} e^{-a\|x\|_\infty^2/n}
    \qquad
    (n \ge \|x\|_\infty)
\end{equation}
(of course  $D^{*n}(x)=0$  when
$n < \|x\|_\infty$).  The heat kernel estimate \refeq{hk} is proved in \cite[Theorem~6.28]{Barl17}.
Unlike local central limit theorems
which give precise constants
(e.g., \cite{LL10}), \refeq{hk} gives an
exponential bound
for $x$ well beyond the diffusive scale, including for $|x|$ and $n$ of comparable size.

\begin{prop}
\label{prop:srw}
For $d>2$, there are constants $a_0>0$ and $a_1\in (0,1)$ such that
for all $\mu \in (0,\mu_c]$,
\begin{equation}
\lbeq{srwbd}
    C_\mu(x) \le a_0 \frac{1}{1\vee |x|^{d-2}}e^{-a_1 m_0(\mu)\|x\|_\infty}
    \qquad
    (x\in\Z^d).
\end{equation}
\end{prop}

\begin{proof}
There is nothing to prove for $x=0$ since $C_{\mu_c}(0)<\infty$, so we assume $x\neq 0$.
We write $\ell=\|x\|_\infty$ and $m_0=m_0(\mu)$.
As in the discussion around \refeq{smallz1}--\refeq{smallz2}
(now with $C_\mu(0)$ in place of $\bubble(z)$), it suffices to consider
$\mu \in [\frac{1}{\degree}-\delta,\frac{1}{\degree}]$ for any small $\delta>0$.
For small enough $\delta$, there is a $c>0$ such that
$\mu\degree =1-(1-\mu\degree) \le
e^{-(1-\mu\degree)} \le
e^{-cm_0^2}$ by \refeq{m0asy}, so from \refeq{hk}
we obtain
\begin{align}
\lbeq{Cmuhk}
    C_\mu(x)
    & = \sum_{n=\ell}^\infty (\mu\degree)^n D^{*n}(x)
     \le A\sum_{n=\ell}^\infty e^{-cm_0^2 n}\frac{1}{n^{d/2}}e^{-a \ell^2/n}.
\end{align}
We apply the inequality $u^2+v^2 \ge 2uv$ with $u^2=cm_0^2n$ and $v^2=a \ell^2/(2n)$,
to obtain
(with $a_1=\sqrt{2ca}$)
\begin{align}
    C_\mu(x) & \le
    Ae^{-a_1m_0\ell} \sum_{n=\ell}^\infty  \frac{1}{n^{d/2}}e^{-a\ell^2/(2n)}.
\end{align}
The terms on the right-hand side increase while $n < a\ell^2/d$ and then decrease,
so by crudely bounding as in the proof of the integral test for series convergence,
\begin{align}
    \sum_{n=\ell}^\infty  \frac{1}{n^{d/2}}e^{-a\ell^2/(2n)}
    & \le
    \sum_{\ell \le n< a\ell^2/d}  \frac{1}{n^{d/2}}e^{-a\ell^2/(2n)}
    +
    \sum_{n\ge a\ell^2/d}  \frac{1}{n^{d/2}}
    \nnb & \le
    2\int_\ell^\infty \frac{1}{t^{d/2}}e^{-a\ell^2/(2t)}dt + O(\ell^{-(d-2)}).
    \nnb
    & =
    2 \frac{1}{\ell^{d-2}}
    \int_0^{\ell}   s^{(d-4)/2} e^{-as/2} ds + O(\ell^{-(d-2)})
    \qquad (\text{substitution $s=\ell^2/t$}).
\lbeq{EMbd}
\end{align}
Since the integral over $[0,\infty)$
converges, this gives an $O(\ell^{-(d-2)})$ upper bound and thus
proves \refeq{srwbd}.
\end{proof}

\subsection{Massive infrared bound}

Let $\mathbb{T}^d = (\R / 2 \pi \Z )^d$ denote\footnote{Not to be
confused with the discrete torus $\T_r^d$ in Section~\ref{sec:torus},
which will not reappear until Section~\ref{sec:plateaupf}.}
the continuous torus.
The Fourier transform of an absolutely summable function $f:\Z^d \to \C$ is defined by
\begin{equation}
    \hat f(k) = \sum_{x\in\Z^d}f(x)e^{ik\cdot x}
    \qquad
    (k \in \T^d),
\end{equation}
and the inverse Fourier transform is
\begin{equation}
    f(x) = \int_{\T^d}\hat{f}(k) e^{-i k \cdot x } \frac{dk}{(2\pi)^{d}}
    \qquad
    (x \in \Z^d).
\end{equation}
In particular, the transform of the step distribution $D$ is
$\hat D(k) = d^{-1}\sum_{j=1}^d \cos k_j$.
We define $A_\mu:\Z^d \to \R$ by
$A_\mu = \delta -\mu\degree D$, with $\delta$ the Kronecker delta
$\delta(x) = \delta_{0,x}$.
It follows from the definition of $C_\mu(x)$ in \refeq{Czdef}
that $C_\mu(x)=\delta_{0,x} + \mu\degree (D*C_\mu)(x)$, so
\begin{equation}
\lbeq{CAdelta}
    C_\mu * A_\mu=\delta.
\end{equation}
Since the Fourier transform converts convolutions to products, this implies that
\begin{equation}
\lbeq{Chatk}
    \hat C_\mu(k) = \frac{1}{\hat A_\mu(k)} = \frac{1}{1-\mu\degree\hat D(k)}.
\end{equation}

Throughout the paper, for $m \ge 0$ and $f:\Z^d \to \C$ we write
$f^{(m)}$ for the \emph{exponential tilt} of $f$:
\begin{equation}
    f^{(m)}(x) = f(x) e^{mx_1}
    \qquad (x=(x_1,\ldots,x_d)).
\end{equation}
Also, for a multi-index $\alpha=(\alpha_1,\ldots,\alpha_d)$
with each $\alpha_i \in \{0,1,2,\ldots\}$, we write $|\alpha|
=\sum_{j=1}^d \alpha_j$.
We will use the fact that $|k|^p+m_0^p \asymp (|k|+m_0)^p$  for any fixed $p\in\N$.

The next lemma is a massive infrared bound.
The purpose of its factor $\sigma$ is to keep
$m$ bounded away from $m_0(\mu)$, so that
the tilt in $C_\mu^{(m)}(x)$ does not remove all of the exponential decay from $C_\mu(x)$.
The decay remaining in $C_\mu^{(m)}(x)$ has rate proportional to $m_0(\mu)$
and is responsible for the $m_0$ term on the right-hand side of \refeq{Cirbd}.

\begin{lemma}
\label{lem:Ck}
Let $d>2$.
Fix $\sigma \in (0,1)$.
For any multi-index $\alpha$ with $|\alpha| \ge 0$,
there is a constant (depending on $\alpha,\sigma$) such that
for all $\mu \in [\frac{1}{2\degree},\frac{1}{\degree})$
and $m \in [0,\sigma m_0(\mu)]$,
\begin{equation}
\lbeq{Cirbd}
    |\nabla^\alpha \hat C^{(m)}_\mu(k)|
    \le
    {\rm const}
    \frac{1}{(|k| +m_0(\mu) )^{2+|\alpha|}}
    \qquad
    (k \in \T^d)
    .
\end{equation}
\end{lemma}

\begin{proof}
It follows from \refeq{CAdelta} that $C^{(m)}_\mu * A_\mu^{(m)}=\delta$.  Since
$\hat A^{(m)}_\mu(k) = 1-\mu\degree \hat D^{(m)}(k)$, we obtain
\begin{align}
\lbeq{CAhat}
    \hat C^{(m)}_\mu(k) & =
    \frac{1}{\hat A^{(m)}_\mu(k)} =
    \frac{1}{\hat A^{(m)}_\mu(0) + \mu\degree [\hat D^{(m)}(0) - \hat D^{(m)}(k)]}
    .
\end{align}
By definition,
\begin{align}
\lbeq{Dk}
    \hat D^{(m)}(k) & =
    \sum_{x\in \Z^d} D(x) e^{mx_1}e^{ik\cdot x}
    =
    \frac{1}{2d} \sum_{x\in\Z^d: |x|=1} e^{mx_1+ik\cdot x}
    \nnb & =
    i\frac 1d  \sinh m  \sin k_1 + \frac 1d  \cosh m \cos k_1 + \frac 1d  \sum_{j=2}^d \cos k_j
    ,
\end{align}
and hence
\begin{align}
    \hat D^{(m)}(0) - \hat D^{(m)}(k) & =
    -i\frac 1d   \sinh m \sin k_1 + \frac 1d  \cosh m(1-\cos k_1) + \frac 1d  \sum_{j=2}^d (1-\cos k_j )
    .
\end{align}
The origin of the formula for
the mass $m_0=m_0(\mu)$ in \refeq{m0def} is that it satisfies $\hat C^{(m_0)}_\mu(0)=\infty$, i.e.,
\begin{equation}
    0 = \hat A^{(m_0)}_\mu(0)
    = 1-2\mu [\cosh m_0  + d-1] ,
\end{equation}
an identity which can be verified from \refeq{m0def}.
Therefore, since $m \in [0,\sigma m_0(\mu)]$ and $\mu \ge \frac{1}{2\degree}$,
there is a constant depending on $\sigma$ such that
\begin{align}
\lbeq{Ahatlb}
    \hat A^{(m)}_\mu(0) &= \hat A^{(m)}_\mu(0) - \hat A^{(m_0)}_\mu(0)
    \nnb &
    = 2\mu [\cosh m_0 - \cosh m]
    \ge {\rm const}\, m_0^2.
\end{align}
Also,
\begin{align}
\lbeq{DDlb}
    | \hat D^{(m)}(0) - \hat D^{(m)}(k)| & \ge
    {\rm Re}\, [\hat D^{(m)}(0) - \hat D^{(m)}(k)]
    \nnb & \ge
    \frac 1d \sum_{j=1}^d (1-\cos k_j) \ge {\rm const}\, |k|^2.
\end{align}
Therefore, by using \refeq{Ahatlb}--\refeq{DDlb}
together with \refeq{CAhat} we obtain
\begin{equation}
\lbeq{Albd}
    |\hat A_\mu^{(m)}(k)| \ge {\rm const}( m_0^2 + |k|^2 ) \ge {\rm const}( m_0 + |k| )^2 .
\end{equation}
This proves the $\alpha = 0$ case of \refeq{Cirbd}.

For $|\alpha| \ge 1$, explicit differentiation of \refeq{Dk} gives
\begin{equation}
\lbeq{Aders}
    |\nabla^\alpha \hat A_\mu^{(m)}(k)| \le {\rm const} \times
    \begin{cases}
        (|k|+m) & (|\alpha| \; \text{odd})
        \\
        1 & (|\alpha| \; \text{even}).
    \end{cases}
\end{equation}
The significance of \refeq{Aders} is for small $k,m$.
Mixed partial derivatives of $\hat D^{(m)}(k)$ and hence of $\hat A_\mu^{(m)}(k)$
all vanish.
As examples,
we compute the first few derivatives (dropping the superscript $(m)$)
using
\refeq{Albd}--\refeq{Aders}:
\begin{align}
\lbeq{Cder12}
    \left| \nabla_i \frac{1}{ \hat A} \right|
    &
    = \left| \frac{\nabla_i\hat A }{\hat A^2} \right|
    \le  c_1 \frac{1}{ (|k|+m_0)^{3}},
    \qquad
    \left| \nabla_i^2 \frac{1}{\hat A} \right|
    \le \left| \frac{\nabla_i^2\hat A }{\hat A^2} \right|
    +
    \left| \frac{2( \nabla_i\hat A)^2 }{ \hat A^3} \right|
    \le c_2 \frac{1}{ (|k|+m_0)^{4}} ,
\end{align}
\begin{align}
\lbeq{Cder3}
    \left| \nabla_i^3 \frac{1}{\hat A} \right|
    &
    \le \left| \frac{\nabla_i^3\hat A }{\hat A^2} \right|
    +
    \left| \frac{6( \nabla_i\hat A)(\nabla_i^2 \hat A) }{ \hat A^3} \right|
    +
    \left| \frac{6( \nabla_i\hat A)^3 }{ \hat A^4} \right|
    \le c_3 \frac{1}{ (|k|+m_0)^{5}}  .
\end{align}
In higher order derivatives, differentiation of the terms arising from the quotient rule
at the previous order are of two types.  When a denominator is differentiated, an
additional factor $(\nabla_i \hat A)/\hat A$ is produced, and this worsens the upper
bound by a factor $(|k|+m_0)^{-1}$.  When a numerator is differentiated,
it either worsens the upper bound by a factor $(|k|+m_0)^{-1}$ (if an odd order derivative
in the numerator is differentiated) or it improves the upper bound by a factor
$(|k|+m_0)^{+1}$ (if an even order derivative in the numerator
is differentiated).  Thus, as claimed in \refeq{Cirbd},
we obtain
\begin{align}
    \left| \nabla_i^n \frac{1}{ \hat A} \right|
    &
    \le  {\rm const} \frac{1}{ (|k|+m_0)^{n+2}}
    .
\end{align}
This completes the proof.
\end{proof}

\section{Lace expansion}
\label{sec:le}

The lace expansion was introduced by Brydges and Spencer \cite{BS85} to prove
that the weakly self-avoiding walk is diffusive in dimensions $d>4$.
In the decades since 1985, the lace expansion has been adapted and extended to a broad range of
models and results \cite{Slad06,HH17book}.

We restrict attention henceforth to dimensions $d>4$ and sufficiently small $\beta>0$.
For the weakly self-avoiding walk, the lace expansion
\cite{BS85,MS93,Slad06}
produces an explicit formula for the $\Z^d$-symmetric function $\Pi_z: \Z^d \to \R$
which for $z\in [0,z_c)$ satisfies the convolution equation
\begin{equation}
\lbeq{lace}
    G_z(x) = \delta_{0,x} + z\degree (D*G_z)(x) + (\Pi_z* G_z)(x) \qquad (x \in \Z^d),
\end{equation}
or equivalently,
\begin{equation}
\lbeq{lacek}
    \hat G_z(k) = \frac{1}{1-z\degree \hat D(k) - \hat \Pi_z(k)}
    \qquad (k \in \T^d).
\end{equation}
We define
\begin{equation}
\lbeq{FFhat}
    F_z  = \delta - z\degree D  - \Pi_z , \qquad \hat F_z = 1 -z\degree\hat D -\hat\Pi_z.
\end{equation}
Then \refeq{lace}--\refeq{lacek} simplify to
\begin{equation}
\lbeq{GstarF}
     G_z * F_z  = \delta , \qquad \hat G_z(k) = \frac{1}{\hat F_z(k)}.
\end{equation}

In fact, $\Pi_z$ is given by an alternating series $\Pi_z(x)=\sum_{N=1}^\infty (-1)^N
\Pi_z^{(N)}(x)$, with each $\Pi_z^{(N)}(x)$ nonnegative and monotone increasing in $z$.
It is proven, e.g., in \cite{BHK18}
that there is a constant $K$ such that
\begin{equation}
\lbeq{PiNbd}
    \Pi_z^{(N)}(x) \le (K\beta)^N \frac{1}{1+|x|^{3(d-2)}}
    \qquad
    (N \ge 1, \, z \in [0,z_c],\, x\in \Z^d).
\end{equation}
Consequently, for any $s < 2d-6$, there is a constant $K_s$ such that
\begin{equation}
\lbeq{xsPi}
    \sum_{x\in \Z^d} |x|^s |\Pi_z(x) |
    \le
    \sum_{x\in\Z^d}\sum_{N=1}^\infty
    |x|^s \Pi_z^{(N)}(x) \le K_s\beta
    \qquad
    (z \in [0,z_c]).
\end{equation}
The inequality \refeq{xsPi} is often referred to as a \emph{diagrammatic estimate}
since it is motivated by a diagrammatic representation of $\Pi_z$ (see
\cite{BS85} or \cite[Section~5.4]{MS93}).
A similar diagrammatic estimate (as in \cite[Theorem~5.4.4]{MS93}) gives
\begin{equation}
\lbeq{dzPi}
    \sum_{x\in \Z^d} |\partial_z \Pi_z(x) |
    \le\sum_{x\in\Z^d}\sum_{N=1}^\infty
    \partial_z \Pi_z^{(N)}(x) \le K' \beta
    \qquad
    (z \in [0,z_c]).
\end{equation}

Since $\chi(z)=\hat G_z(0)$ and
$\chi(z_c)=\infty$, we have $0=\hat F_{z_c}(0) = 1-z_c\degree -\hat\Pi_{z_c}(0)$.
By \refeq{xsPi} with $s=0$, this implies that
\begin{equation}
\lbeq{zcasy}
    z_c - \frac{1}{\degree} \le K_0 \beta.
\end{equation}
With $s=1$, it also implies that
there is a $z^* \in (z,z_c)$ such that
\begin{align}
    \chi(z)^{-1} &= \hat F_z(0) = \hat F_z(0) - \hat F_{z_c}(0)
    \nnb
    & = \degree(z_c-z) + \hat\Pi_{z_c}(0)-\hat \Pi_z(0)
    \nnb & =
    \degree(z_c-z) + \partial _z\hat\Pi_{z}(0)|_{z=z^*} (z_c-z),
\end{align}
using the mean-value theorem for the last equality.
It follows from \refeq{dzPi} and the dominated convergence theorem
that the coefficient in the last term approaches $\partial _z\hat\Pi_{z}(0)|_{z=z_c}
=O(\beta)$ as $z \to z_c$.  This proves that
\begin{equation}
\lbeq{chiasy}
    \chi(z) \sim A(1-z/z_c)^{-1} \qquad (z \to z_c)
\end{equation}
with $A^{-1}=-z_c\partial_z\hat F_{z_c}(0)=z_c (\degree + \partial_z\hat \Pi_{z_c}(0))
=1 + O(\beta)$, using \refeq{dzPi}--\refeq{zcasy} in the last equality.

\section{Asymptotic formula for the mass}
\label{sec:masy}

In this section, we prove the statement from \refeq{masy} that for
$d>4$ and sufficiently small $\beta>0$,
\begin{equation}
\lbeq{masy2}
    m(z) \sim {\rm const}\,(1-z/z_c)^{1/2}  \qquad   (z \to z_c),
\end{equation}
with constant equal to $\degree^{1/2}+O(\beta)$.
The essence of the proof is as in \cite[Section~6.5]{MS93}
(originally in \cite{HS92a}), which itself
is based on \cite{Hara90}.

Let $z<z_c$, $m<m(z)$, and $\chi^{(m)}(z)=\sum_{x\in\Z^d}G_z^{(m)}(x)$.
The tilted version of \refeq{GstarF} is $G^{(m)}_z*F^{(m)}_z = \delta$
and hence
\begin{equation}
\lbeq{lacemk}
    \hat G_z^{(m)}(k) = \frac{1}{\hat F_z^{(m)}(k)}
    =
    \frac{1}{1-z\degree \hat D^{(m)}(k) - \hat \Pi_z^{(m)}(k)}.
\end{equation}
In particular
(recall \refeq{Dk}),
\begin{equation}
\lbeq{chimz}
    \frac{1}{\chi(z)}-\frac{1}{\chi^{(m)}(z)}
    =
    \hat F_z(0) - \hat F_z^{(m)}(0)
    = 2z [\cosh m -1] + [\hat\Pi_z^{(m)}(0)-\hat\Pi_z(0)].
\end{equation}
An argument\footnote{Briefly,
if $\chi(m(z))$ were finite then $G_z^{(m(z))}$ would decay
exponentially and this contradicts the fact that $m(z)$ is by \refeq{mzdef} the
exponential decay rate of $G_z$.}
based on the Lieb--Simon
inequality gives $\chi^{(m)}(z) \to \infty$ as $m \to m(z)$,
exactly as in the proof of \cite[(6.5.7)]{MS93}.
Therefore, when we take the limit $m \to m(z)$ (from the left) in \refeq{chimz} we get
\begin{align}
\lbeq{chimz1}
    \frac{1}{\chi(z)}
    & = 2z [\cosh m(z) -1] + [\hat\Pi_z^{(m(z))}(0)-\hat\Pi_z(0)]
    ,
\end{align}
provided we can justify that the
limit $\lim_{m \to m(z)}\hat\Pi_z^{(m)}(0)=\hat\Pi_z^{(m(z))}(0)$
exists.
The next proposition takes care of this last point, via dominated convergence.
The proof of Proposition~\ref{prop:Pim} is deferred to later in this section.

\begin{prop}
\label{prop:Pim}
Let $d>4$ and let $\beta$ be sufficiently small.
Let\footnote{We expect that the proposition in fact remains true for all $s \in [0, 2d-6)$ but its
restriction to $s \in [0,d-2]$ is sufficient for our needs.} $s \in [0,d-2]$.
There is a constant $K_s'$ such that, uniformly in $z \in [\frac{1}{\degree},z_c)$,
\begin{equation}
\lbeq{Pims}
    \sum_{x\in\Z^d}  |x|^s  \sum_{N=1}^\infty |\Pi_z^{(N,m(z))} (x)| \le K_s' \beta.
\end{equation}
\end{prop}

By \refeq{chiasy},
the left-hand side of \refeq{chimz1} is asymptotic to $A^{-1}(1-z/z_c)$ as $z\to z_c$.
Since $m(z) \to 0$, the first term on the right-hand side of \refeq{chimz1}
is asymptotic to $z_c m(z)^2$.
Given $\epsilon \in (0,2]$, it is an exercise with the Maclaurin series for $\cosh t$ to prove that
\begin{equation}
\lbeq{coshbd}
    |\cosh t -1 - \frac 12 t^2| \le {\rm const}\, t^{2+\epsilon} \cosh t.
\end{equation}
We apply \refeq{coshbd} with $\epsilon = \min\{2,d-4\}$, so $2+\epsilon \le d-2$,
and  conclude from Proposition~\ref{prop:Pim} and the fact that $\Pi_z(x)=\Pi_z(-x)$
that the last term on the right-hand side
 of \refeq{chimz1} is
\begin{align}
    \hat\Pi_z^{(m(z))}(0)-\hat\Pi_z(0)
    &
    = \sum_{x\in \Z^d} (\cosh(m(z) x_1)-1 ) \Pi_z(x)
    \nnb & =
    \frac 12 m(z)^2  \sum_{x\in\Z^d} x_1^2 \Pi_z(x)
    + O(m(z)^{2+\epsilon})
    .
\end{align}
Since $m(z) \to 0$ as $z \to z_c$, we conclude from the above that
\begin{align}
\lbeq{Am2}
    A^{-1}(1-z/z_c)  & \sim
    m(z)^2\Big[ z_c + \frac 12 \sum_{x\in\Z^d} x_1^2 \Pi_{z_c}(x)  \Big].
\end{align}
This gives,
as desired,
\begin{align}
     m(z)^2 \sim c (1-z/z_c) \qquad (z \to z_c).
\end{align}
From \refeq{Am2}, we observe that the constant
$c$ is given by
 $c =  A^{-1} [z_c + \frac 12 \sum_{x\in\Z^d} x_1^2 \Pi_{z_c}(x)]^{-1} = \degree +O(\beta)$,
 since $A=1+O(\beta)$ (see below \refeq{chiasy}) and $z_c = \degree^{-1}+O(\beta)$ by
 \refeq{zcasy}.

It remains to prove Proposition~\ref{prop:Pim}, which we will do using  the
following lemma.
Its hypothesis involves the tilted bubble diagram
\begin{equation}
\lbeq{Bmdef}
    \bubble^{(m)}(z) = \sum_{x\in \Z^d} G_z^{(m)}(x)^2.
\end{equation}

\begin{lemma}
\label{lem:diagram-m}
Fix $z \in [\frac{1}{\degree},z_c)$ and $m\ge 0$.  Let $s \in [0,d-2]$.
Suppose that there is a constant $\kappa$ such that $\bubble^{(m)}(z) \le \kappa$.
Then there is a constant $K_s'$ depending on $\kappa$ (and not on $m,z$) such that
for sufficiently small $\beta$ (depending on $\kappa$),
\begin{equation}
\lbeq{Pibdsum}
    \sum_{x\in\Z^d} |x|^s \sum_{N=1}^\infty \Pi_z^{(N,m)}(x) \le K_s'\beta.
\end{equation}
\end{lemma}

\begin{proof}
This is a standard diagrammatic estimate, just as in \cite[Corollary~6.5.2]{MS93}.
We therefore only sketch the argument.
The estimate is however better than \cite[Corollary~6.5.2]{MS93}
by allowing a larger range of $s$, as a consequence of \refeq{Gzc} which
was not available when \cite[Corollary~6.5.2]{MS93} was proved.

We illustrate the idea of the proof with the 4-loop term $\Pi^{(4)}$,
which obeys the estimate
\begin{equation}
\lbeq{Pi4}
    \Pi_z^{(4)}(x)
    \le \beta^4 \sum_{u,v\in\Z^d} G_z(u)^2 G_z(v) G_z(u-v) G_z(x-u) G_z(x-v)^2.
\end{equation}
The proof of \refeq{Pi4} is a small modification of the proof of \cite[Theorem~5.4.2]{MS93}
(the small parameter $\beta$ is more explicit here and this is a simplification).
With an exponential tilt, we obtain
\begin{equation}
\lbeq{Pi4m}
    \Pi_z^{(4,m)}(x) \le \beta^4 \sum_{u,v\in\Z^d} G_z(u)G_z^{(m)}(u) G_z(v) G_z(u-v) G_z^{(m)}(x-u) G_z(x-v)^2.
\end{equation}
We multiply by $|x|^s$ and sum over $x\in\Z^d$,
use the inequality $|x|^s \le 2^s(|v|^s + |x-v|^s)$, and
consider the effect of each of the two terms on the right-hand side for the resulting sum over $x$.  For example, one of these two terms is
\begin{equation}
    T =
    \beta^4
    \sum_{x,u,v\in\Z^d} G_z(u)G_z^{(m)}(u) |v|^s G_z(v) G_z(u-v) G_z^{(m)}(x-u) G_z(x-v)^2.
\end{equation}
It follows from \cite[Lemma~5.4.3]{MS93} that the sum can be bounded by the product
of $\sup_v |v|^s G_z(v)$ times the $\ell_2$ norm of each of the other six factors $G_z$
or $G_z^{(m)}$.  By hypothesis, those $\ell_2$ norms are at most $\kappa^{1/2}$, and the supremum is bounded
by a constant due to \refeq{Gzcub} and the restriction $s \le d-2$.
The result is a bound of order  $\beta^4 \kappa^3$.

In more detail, by definition of $T$,
\begin{align}
    T & \le \beta^4
    \left(\sup_v |v|^s G_z(v)\right)
    \sup_u \left(\sum_{x,v} G_z(u-v) G_z^{(m)}(x-u) G_z(x-v)^2 \right)
    \left( \sum_u  G_z(u)G_z^{(m)}(u) \right)
     .
\end{align}
The first factor on the right-hand side is bounded via \refeq{Gzcub}.
The last factor is bounded by $\kappa$, by the Cauchy--Schwarz inequality and by hypothesis.
The middle factor, after replacement of the summation indices by $x\mapsto x+u$ and $v
\mapsto v+u$, followed by replacement of $x$ by $y=x-v$, becomes
\begin{align}
    \sum_{x,v} G_z(v) G_z^{(m)}(x) G_z(x-v)^2
    & =
    \sum_{v,y} G_z(v) G^{(m)}_z(y+v) G_z(y)^2
    \nnb & \le
    \left( \sup_y \sum_v G_z(v) G^{(m)}_z(y+v) \right) \sum_y G_z(y)^2
    \le
    \kappa ^2
     ,
\end{align}
where we used the Cauchy--Schwarz inequality for the supremum.
Overall, this gives a bound on $T$ of order $\beta^4 \kappa^3$.  This procedure can be
mechanised for general $N$ as in \cite[Section~5.4]{MS93}.

As a result, for general $N$, $\Pi^{(N)}$ can be bounded similarly by factoring
the weight $e^{mx_1}$ along one side of the $N$-loop diagram and factoring $|x|^s$
along the other side of the diagram.
The bound on the $N^{\rm th}$ term in \refeq{Pibdsum} is of order $N\beta^N \kappa^{N-1}$, where the factor $N$
accounts for the use of the triangle inequality to distribute $|x|^s$ along fewer than $N$
diagram lines.  Summation over $N$ gives \refeq{Pibdsum},
for $\beta$ small depending on $\kappa$.
\end{proof}

\begin{proof}[Proof of Proposition~\ref{prop:Pim}]
Fix $z \in [\frac{1}{\degree},z_c)$.
We define the critical simple random walk bubble diagram
\begin{equation}
\lbeq{B0def}
    \bubble_0 = \sum_{x\in \Z^d} C_{1/\degree}(x)^2
    = \int_{\T^d} \frac{1}{(1-\hat D(k))^2} \frac{dk}{(2\pi)^d} < \infty,
\end{equation}
where the second inequality holds by the Parseval relation and \refeq{CAdelta}.
By Lemma~\ref{lem:diagram-m}, it suffices to prove that
\begin{equation}
\lbeq{Bm}
    \bubble^{(m(z))}(z)  \le 2\bubble_0.
\end{equation}
We prove \refeq{Bm} with a bootstrap argument.
In many applications of the lace expansion, the bootstrap argument is used to produce a forbidden
interval for the parameter $z$.  Here, we instead produce a forbidden interval
for the mass parameter $m$; this strategy was first used in \cite{Hara90} for
percolation and was subsequently applied also to self-avoiding walk in \cite{HS92a,MS93}.

To prove \refeq{Bm}, we will prove that:
\\
\indent (i) $\bubble^{(0)}(z) \le \frac 32 \bubble_0$,
and
\\
\indent
(ii) if for $m \in(0, m(z))$ we assume $\bubble^{(m)}(z) \le 3 \bubble_0$ then in
fact $\bubble^{(m)}(z) \le 2 \bubble_0$.
\\
By (ii), the interval $(2\bubble_0,3\bubble_0]$ is forbidden for values of $\bubble^{(m)}(z)$
when $m\in(0, m(z))$.
By (i), $\bubble^{(0)}(z)$ lies below this forbidden interval.
Since $\bubble^{(m)}(z)$ is continuous in $m$ by monotone convergence, it follows
from the Intermediate Value Theorem that
$\bubble^{(m)}(z) \le 2\bubble_0$ for all $m<m(z)$.  By monotone convergence, this
implies \refeq{Bm}.  So it remains to prove (i) and (ii).

The main work is to prove (ii); along the way we also prove (i).
By \refeq{Bmdef}, \refeq{lacemk} and the Parseval relation, for $m\in [0,m(z))$ we have
\begin{equation}
\lbeq{Bmk}
    \bubble^{(m)}(z)
    = \int_{\T^d} \frac{1}{|\hat F_z^{(m)}(k)|^2} \frac{dk}{(2\pi)^d}.
\end{equation}
The denominator of the integrand involves
\begin{align}
\lbeq{F0k}
    \hat F^{(m)}_z(k) & = \hat F^{(m)}_z(0) + [\hat F^{(m)}_z(k) - \hat F^{(m)}_z(0)].
\end{align}
Thus, by the formula for $\hat F^{(m)}(k)$ in \refeq{lacemk},
\begin{align}
    |\hat F^{(m)}_z(k)| & \ge {\rm Re}[\hat F^{(m)}_z(k) - \hat F^{(m)}_z(0)]
    \nnb & =
    z\degree  {\rm Re}\, [\hat D^{(m)}(0) - \hat D^{(m)}(k)]
    +
    {\rm Re}\, [\hat \Pi_z^{(m)}(0) - \hat \Pi_z^{(m)}(k)],
\end{align}
where the inequality holds since $\hat F^{(m)}_z(0)$ is real and positive.
Since $z\degree \ge 1$ by assumption and because
${\rm Re}\, [\hat D^{(m)}(0) - \hat D^{(m)}(k)]
\ge 1-\hat D(k)$
by \refeq{DDlb},
\begin{align}
    |\hat F^{(m)}_z(k)| & \ge
    1-\hat D(k)
    +
    \sum_{x\in \Z^d} (1-\cos (k\cdot x)) \Pi^{(m)}_z(x)
    \nnb & \ge
    1-\hat D(k) -
    \sum_{x\in \Z^d}  (1-\cos (k\cdot x)) |\Pi_z^{(m)}(x)|
    .
\lbeq{Fmlb}
\end{align}
To bound the last sum in \refeq{Fmlb}, we use the facts that
$1-\cos t \le \frac 12 t^2$ for $t\in \R$,
and
$2\pi^{-2}t^2 \le 1-\cos t$ for $t\in [-\pi,\pi]$.
Since $|k\cdot x | \le |k|\,|x|$,
\begin{align}
    \sum_{x\in \Z^d} (1-\cos (k\cdot x)) |\Pi_z^{(m)}(x)|
    & \le
    \frac{1}{2} |k|^2 \sum_{x\in \Z^d} |x|^2 |\Pi_z^{(m)}(x)|
    \nnb & \le
    [1-\hat D(k)] \frac{d\pi^2}{4} \sum_{x\in \Z^d} |x|^2 |\Pi_z^{(m)}(x)|
    .
\lbeq{Fmlb1}
\end{align}
By Lemma~\ref{lem:diagram-m},
if we assume that $\bubble^{(m)}(z) \le 3\bubble_0$ (as we do for $m\in(0,m(z))$ in (ii))
then
the right-hand side of \refeq{Fmlb1} is $O(\beta)[1-\hat D(k)]$,
so we obtain
\begin{align}
\lbeq{Fmlb2}
    |\hat F^{(m)}_z(k)| &
    \ge {\rm Re}[\hat F^{(m)}_z(k) - \hat F^{(m)}_z(0)]
    \ge
    [1-O(\beta)](1-\hat D(k)).
\end{align}
But this implies the infrared bound
\begin{equation}
\lbeq{Gmirbd0}
    |\hat G^{(m)}_z(k)| \le (1+O(\beta)) \frac{1}{1-\hat D(k)},
\end{equation}
which by \refeq{Bmk} and \refeq{B0def}
implies that $\bubble^{(m)}(z) \le 2 \bubble_0$, and completes the proof of item (ii).

Finally, we prove (i).
For (i) we do not have an {\it a priori} assumption that
$\bubble^{(0)}(z) \le 3\bubble_0$ as was used the previous paragraph.
Nevertheless, the proof of (ii) more than suffices for the proof of (i).
Indeed, the above argument also implies \refeq{Gmirbd0} when $m=0$, using
\refeq{xsPi} to bound the right-hand side  of \refeq{Fmlb1} instead of making
use of the {\it a priori} assumption, and
this implies that $\bubble^{(0)}(z) \le \frac 32 \bubble_0$.
This completes the proof of both items (i) and (ii) in the
bootstrap argument, and concludes the proof.
\end{proof}

Note that \refeq{Fmlb2}, \refeq{Gmirbd0} and \refeq{Pibdsum}, although initially conditional on the bootstrap hypothesis,
hold unconditionally now that the bootstrap argument has been completed.
In particular, we have now proved that \refeq{Bm} does in fact hold,
and hence for $z \in [\frac{1}{\degree},z_c)$ we have
\begin{equation}
\lbeq{Bm-bis}
    \bubble^{(m(z))}(z)  \le 2\bubble_0,
\end{equation}
which is a statement of a tilted bubble condition.

\section{Proof of Theorem~\ref{thm:mr}}
\label{sec:pfmr1}

We now complete the proof of Theorem~\ref{thm:mr} by proving the inequality \refeq{Gmr}
for $G_z(x)$.  The proof is based on an extension of the method of \cite{Slad20_lace}.

\subsection{Isolation of leading term}

We extend the method of \cite{Slad20_lace} to include positive mass.
Let $\lambda  > 0$, $\mu\in[0,\frac{1}{\degree}]$, and $A_\mu = \delta -\mu\degree D$.
Since $C_\mu * A_\mu=\delta$ and $F_z *G_z = \delta$ by \refeq{CAdelta} and \refeq{GstarF},
we have
\begin{align}
    G_z
    & = \lambda C_\mu + \delta * G_z - \lambda C_\mu * \delta
    \nnb &= \lambda C_\mu + C_\mu * E_{z,\lambda,\mu} * G_z
    \quad \text{with} \quad E_{z,\lambda,\mu}= A_\mu - \lambda F_z.
\lbeq{Edef}
\end{align}
As in \cite{HHS03,Slad20_lace}, we choose $\lambda_{z}$ and $\mu_{z}$ so that
\begin{equation}
\lbeq{Ezero}
    \sum_{x\in \Z^d}E_{z}(x)
    =
    \sum_{x\in \Z^d} |x|^2 E_{z}(x)
    =
    0,
\end{equation}
where we set $E_z = E_{z,\lambda_{z},\mu_{z}}$.
The solution to the two linear equations \refeq{Ezero} in the two unknowns $\lambda,\mu$ is
\begin{align}
    \lambda_{z} & = \frac{1}{1-\hat\Pi_z(0) + \sum_x |x|^2 \Pi_z(x)},
\lbeq{lambdaz}
    \\
    \mu_{z}\degree
    & =
    1-\lambda_{z} \hat F_z(0)
    \nnb & =
    \frac{z\degree  + \sum_x |x|^2 \Pi_z(x)}
    {\hat F_z(0) + z\degree  + \sum_x |x|^2 \Pi_z(x)}.
\lbeq{muzm}
\end{align}
The term $\hat F_z(0)=\chi(z)^{-1}$ is positive for $z<z_c$.
We are interested in the case $z \in [\frac{1}{\degree},z_c)$, so
$z\degree \ge 1$.  By \refeq{xsPi}, the $\Pi$ terms in \refeq{lambdaz}--\refeq{muzm}
are small, in particular $\lambda_z=1+O(\beta)$.
Also, the right-hand side of \refeq{muzm} lies in $(0,1)$ and therefore
$\mu_{z}\in (0,\frac{1}{\degree})$ is
subcritical.
Explicit calculation using the definition of $E_z$ from \refeq{Edef}
and of $\mu_{z}$ from \refeq{muzm} leads to
\begin{align}
    E_z &= (1-\lambda_z)(\delta - D) - \lambda_z \hat\Pi_z(0)D
    + \lambda_z \Pi_z,
    \\
\lbeq{Emformula}
        E_z^{(m)} &= (1-\lambda_z)(\delta - D^{(m)}) - \lambda_z \hat\Pi_z(0)D^{(m)}
    + \lambda_z \Pi_z^{(m)}.
\end{align}
Multiplication of \refeq{Edef} by $e^{mx_1}$ gives
\begin{equation}
    G^{(m)}_z = \lambda_z C^{(m)}_{\mu_{z}} + f^{(m)}_z
    \quad\text{with}\quad
    f^{(m)}_z  = C^{(m)}_{\mu_{z}} * E^{(m)}_{z} * G^{(m)}_z .
\lbeq{fm}
\end{equation}
We will show that $\lambda_z C^{(m)}_{\mu_{z}}$ gives the main contribution
to $G^{(m)}_z$, with $f^{(m)}_z$ smaller by a factor $\beta$.

\subsection{The key ingredient}

The key ingredient in the proof of the main result \refeq{Gmr} is the following proposition.

\begin{prop}
\label{prop:h}
Let $d>4$ and let $\beta$ be sufficiently small.
Let $z \in [\frac{1}{\degree},z_c)$ and $m \in [0, \frac 12 m(z)]$.
There is a constant $A_1>0$ (independent of $m,z,\beta$) such that
\begin{equation}
\lbeq{nablafbd}
    \int_{\T^d} |\nabla^\alpha \hat f_z^{(m)}(k)| \frac{dk}{(2\pi)^d}    \le A_1 \beta
    \qquad
    (|\alpha| \le d-2).
\end{equation}
\end{prop}

Before proving Proposition~\ref{prop:h}, we show that it leads to a proof
of \refeq{Gmr} and thereby concludes the proof of Theorem~\ref{thm:mr}.
For this we need to anticipate a conclusion of Lemma~\ref{lem:mm0}, where it
is proved that
if $z \in [\frac{1}{\degree},z_c)$
then
\begin{equation}
\lbeq{mmu}
     \frac 12 m(z) \le \frac 23 m_0(\mu_{z}).
\end{equation}
Here $m_0(\mu)$ is the simple random walk mass defined in \refeq{m0def}; according
to \refeq{m0asy} it vanishes at $\frac{1}{\degree}$ as
$m_0(\mu)^2 \sim \degree (1-\degree\mu)$.

\begin{proof}[Proof of Theorem~\ref{thm:mr}]
Let $z \in [\frac{1}{\degree},z_c)$ and
set $\tilde m =  \frac 12 a_1 m(z)$, with $a_1<1$ the constant of Proposition~\ref{prop:srw}.
Then $\tilde m \le \frac 12 m(z)$.
We use the fact that the smoothness of the Fourier transform $\hat f_z^{(m)}$
controls the decay of $f_z^{(m)}$, in the sense that  Proposition~\ref{prop:h}
implies that
$|f_z^{(\tilde m)}(x)| \le  A_1'\beta (1\vee |x|^{d-2})^{-1}$.
The proof of this, which is simply integration by parts, can be found in
\cite[Corollary~3.3.10]{Graf14}.
Therefore, by \refeq{fm},
\begin{align}
\lbeq{GCbd}
    G_z^{(\tilde m)}(x) & \le \lambda_{z} C_{\mu_{z}}^{(\tilde m)}(x)  + A_1'\beta
    \frac{1}{1\vee |x|^{d-2}}.
\end{align}
By Proposition~\ref{prop:srw},
\begin{equation}
\lbeq{Cmubd}
    C_{\mu_{z}}^{(\tilde m)}(x)
    = C_{\mu_z}(x) e^{\tilde mx_1}
    \le a_0 \frac{1}{1\vee |x|^{d-2}} e^{-(a_1m_0(\mu_{z})-\tilde m)\|x\|_\infty}.
\end{equation}
By \refeq{mmu},
$\tilde m = \frac 12 a_1 m(z) \le   \frac 23 a_1 m_0(\mu_{z})$.
As noted below \refeq{muzm},  $\lambda_z=1+O(\beta)$.
With \refeq{GCbd}--\refeq{Cmubd},
this shows that by taking $\beta$ sufficiently small we can obtain
\begin{equation}
\lbeq{Gbd}
    G_z^{(\tilde m)}(x) \le (1+O(\beta))a_0\frac{1}{1\vee |x|^{d-2}} + A_1'\beta \frac{1}{1\vee |x|^{d-2}} \le 2 a_0\frac{1}{1\vee |x|^{d-2}} ,
\end{equation}
and hence
\begin{equation}
    G_z(x) \le 2 a_0\frac{1}{1\vee |x|^{d-2}} e^{-\frac 12 a_1 m(z) x_1}.
\end{equation}
Without loss of generality, by the symmetry of $G_z$ we may assume that $x_1=\|x\|_\infty$.
Since $\|x\|_\infty \ge d^{-1/2}|x|$, this proves that
\refeq{Gmr} holds with $c_0=2a_0$ and $c_1 = \frac 12 a_1 d^{-1/2}$.
This completes the proof.
\end{proof}

\subsection{Proof of Proposition~\ref{prop:h}}

To complete the proof of \refeq{Gmr}, it remains to prove Proposition~\ref{prop:h}
and \refeq{mmu}.
By \refeq{fm},
\begin{equation}
\lbeq{fprod}
    \hat f^{(m)}_z  = \hat C^{(m)}_{\mu_{z}} \hat E^{(m)}_{z} \hat G^{(m)}_z ,
\end{equation}
so to prove Proposition~\ref{prop:h}
we need estimates on derivatives of each of the three factors on the right-hand side of \refeq{fprod}.
Lemmas~\ref{lem:Girbd} and \ref{lem:Ebd} give the estimates we need for $\hat G^{(m)}_z$
and $\hat E^{(m)}_{z}$, respectively.
Lemma~\ref{lem:Ck} will give the required estimate for
$\hat C_{\mu_{z}}^{(m)}$, when combined with the relation between $m_0(\mu_z)$
and $m(z)$ claimed in \refeq{mmu} and established in Lemma~\ref{lem:mm0}.

\begin{lemma}
\label{lem:Girbd}
Let $d>4$ and let $\beta$ be sufficiently small.
Let $z \in [\frac{1}{\degree},z_c)$, $m \in [0, \frac 12 m(z)]$, and $0 \le |\alpha| \le d-2$.
There is a constant (independent of $z,m,k,\beta,\alpha$) such that
\begin{equation}
\lbeq{Girbd}
    |\nabla^\alpha \hat G^{(m)}_z(k)|
    \le
    {\rm const}
    \frac{1}{(|k| +m )^{2+|\alpha|}}
    \qquad (k \in \T^d)
    .
\end{equation}
\end{lemma}

\begin{proof}
We have already proved a weaker version
of the case $\alpha=0$ of \refeq{Girbd} in \refeq{Gmirbd0}, with $m=0$ on the right-hand side.
To improve \refeq{Gmirbd0}, we first observe that
the inequality
\begin{equation}
\lbeq{FFk2}
    {\rm Re}\, [\hat F^{(m)}_z(k) - \hat F^{(m)}_z(0)] \ge {\rm const}\,|k|^2
\end{equation}
follows directly from \refeq{Fmlb2}.
To prove \refeq{Girbd} for $\alpha=0$,
we will prove that  $\hat F^{(m)}_z(0) \ge {\rm const}\, m^2$.  This indeed completes the
proof, since with \refeq{FFk2} it gives
\begin{equation}
\lbeq{Fkm}
    |\hat F^{(m)}_z(k)| \ge {\rm Re}\, F^{(m)}_z(k)
    = \hat F^{(m)}_z(0) +{\rm Re}\, [\hat F^{(m)}_z(k) - \hat F^{(m)}_z(0)]
    \ge {\rm const}\, (m^2 + |k|^2),
\end{equation}
which is the $\alpha=0$ case of \refeq{Girbd}.

The proof that  $\hat F^{(m)}_z(0) \ge {\rm const}\, m^2$ is as follows.
By \refeq{zcasy},
$z_c-\frac{1}{\degree} = O(\beta)$, so by \refeq{masy} $m(z) =O(\beta^{1/2})$
is small.
As noted below \refeq{chimz},  $\hat F_z^{(m(z))}(0)=0$, and hence by \refeq{Dk}
\begin{align}
    \hat F_z^{(m)}(0) &= \hat F_z^{(m)}(0) - \hat F_z^{(m(z))}(0)
    \nnb & =
    z\degree [\hat D^{(m(z))}(0) - \hat D^{(m)}(0)]
    + [\hat \Pi_z^{(m(z))}(0) - \hat \Pi_z^{(m)}(0)]
    \nnb & =
    2z [\cosh m(z) - \cosh m] + \sum_{x\in\Z^d} \big(\cosh (m(z)x_1) - \cosh (mx_1)\big) \Pi_z(x).
\end{align}
For the first term, we use $m\le \frac 12 m(z)$ and the fact that $m(z)=O(\beta^{1/2})$
to obtain
\begin{equation}
    2z [\cosh m(z) - \cosh m] = z(m(z)^2-m^2) +O(m(z)^4)
    = z(m(z)^2-m^2) +O(\beta m(z)^2).
\end{equation}
For the $\Pi$ term we use the elementary inequality
\begin{align}
    0 \le \cosh t - \cosh s \le (t-s)\sinh t \le (t-s)t \cosh t
    \qquad
    (0 \le s \le t)
\end{align}
to conclude that
\begin{align}
    \sum_{x\in\Z^d} \big(\cosh (m(z)x_1) - \cosh (mx_1)\big) |\Pi_z(x)|
    & \le m(z)^2 \sum_{x\in\Z^d} x_1^2   \sum_{N=1}^\infty \Pi_z^{(N,m(z))}(x).
\end{align}
The right-hand side is $O(\beta m(z)^2)$ by
Proposition~\ref{prop:Pim}, so the above leads to
\begin{equation}
\lbeq{F0m2}
    \hat F_z^{(m)}(0) = z(m(z)^2-m^2)  +O(\beta m(z)^2).
\end{equation}
Since $m \le \frac 12 m(z)$ by assumption,
this gives $\hat F_z^{(m)}(0) \ge {\rm const}\, m^2$ and the proof for $\alpha=0$
is complete.

Examples of
the first few derivatives of $\hat G_z^{(m)}$ are, with $\hat F = \hat F_z^{(m)}$,
\begin{align}
\lbeq{Gder12}
    \nabla_i \hat G_z^{(m)}
    &
    = -\frac{\nabla_i\hat F }{\hat F^2},
    \qquad
    \nabla_i^2 \hat G_z^{(m)}
    =
    -\frac{\nabla_i^2\hat F }{\hat F^2}
    +
    \frac{2( \nabla_i\hat F)^2 }{ \hat F^3}  ,
\end{align}
\begin{align}
\lbeq{Gder3}
    \nabla_i^3 \hat G_z^{(m)}
    &
    =
    - \frac{\nabla_i^3\hat F }{\hat F^2}
    +
     \frac{6( \nabla_i\hat F)(\nabla_i^2 \hat F) }{ \hat F^3}
    -
    \frac{6( \nabla_i\hat F)^3 }{ \hat F^4} .
\end{align}
The denominators are bounded using the bound
$\hat  F_z^{(m)}(k) \ge {\rm const}\, (|k|+m)^2$ from \refeq{Fkm},
exactly as in the proof of Lemma~\ref{lem:Ck}, and
we again need to show that large powers in denominators are compensated by the numerators.
This works essentially in the same way as in the proof of Lemma~\ref{lem:Ck}.
For the numerators, as in \refeq{Aders} the $\hat D^{(m)}$ terms are bounded using
\begin{equation}
\lbeq{Dders}
    |\nabla^\alpha \hat D_\mu^{(m)}(k)| \le {\rm const} \times
    \begin{cases}
        (|k|+m) & (|\alpha| \; \text{odd})
        \\
        1 & (|\alpha| \; \text{even}),
    \end{cases}
\end{equation}
and it suffices to prove the same estimate for $\nabla^\alpha \hat \Pi^{(m)}_z$.
In fact, it can be extrapolated from \refeq{Gder12}--\refeq{Gder3} that there is just one
term in $\nabla^\alpha\hat{G}_z^{(m)}$ that requires an estimate on
$\nabla^\alpha \hat \Pi^{(m)}_z$ with $|\alpha|=d-2$, namely a term proportional to
$\hat F^{-2}\nabla_i^\alpha\hat F $ with $|\alpha|=d-2$.  Such a term is relatively
small, since the numerator is bounded by a constant due to Proposition~\ref{prop:Pim},
so this term is $O((|k|+m)^{-4})$ which is better than the required
bound $O((|k|+m)^{-2-|\alpha|})$ when $|\alpha|\ge 3$.
It therefore suffices to prove the analogue of \refeq{Dders}
for $\nabla^\alpha \hat \Pi^{(m)}_z$ only for $|\alpha|\le d-3$.

By symmetry, $\hat \Pi^{(m)}_z$ can be written as
\begin{equation}
    \hat \Pi^{(m)}_z(k) =
    \sum_{x\in\Z^d} \cos(k\cdot x) \cosh(mx_1)\Pi_z(x) + i
    \sum_{x\in\Z^d}\sin(k\cdot x) \sinh (mx_1) \Pi_z(x) .
\end{equation}
To estimates its $k$-derivatives, we use the inequalities
\begin{align}
\lbeq{cosder}
    |\nabla^\alpha \cos (k\cdot x)| &
    \le
    \begin{cases}
    |x|^{|\alpha|} & (\text{$|\alpha|$ even})
    \\
    |x|^{|\alpha|+1}|k| & (\text{$|\alpha|$ odd}),
    \end{cases}
    \\
    |\nabla^\alpha \sin (k\cdot x)| &
    \le
    |x|^{|\alpha|} ,
    \\
    \sinh (mx_1)  &
    \le
    m|x_1|\cosh (mx_1).
\end{align}
The factor $|k|$ in \refeq{cosder} is bounded, so this gives
\begin{align}
    |\nabla^\alpha \hat \Pi^{(m)}_z(k)|
    & \le
    {\rm const} \sum_{x\in\Z^d} |x|^{|\alpha|+1} |\Pi_z^{(m)}(x)| \times
    \begin{cases}
        (|k|+m) & (|\alpha| \; \text{odd})
        \\
        1 & (|\alpha| \; \text{even})  .
    \end{cases}
\end{align}
We have already observed that it is now sufficient to consider $|\alpha| \le d-3$,
and by
Proposition~\ref{prop:Pim},
$\sum_{x\in\Z^d}|x|^s|\Pi^{(m)}_z(x)| \le O(\beta)$ for $s \le d-2$.
This yields the required
analogue of \refeq{Dders} for derivatives of $\hat\Pi_z^{(m)}$, and completes the proof.
\end{proof}

\begin{rk}
\label{rk:perc}
In the proof of Lemma~\ref{lem:Girbd}, we require $d-2$ derivatives of
$\hat\Pi_z^{(m)}$, which are controlled by the $(d-2)^{\rm nd}$ moment of $\Pi_z^{(m)}$.
The latter ultimately goes back to
\refeq{xsPi} with $s=d-2$, which itself requires
 $d+(d-2)<3(d-2)$, i.e., $d>4$.
For the Ising and $\varphi^4$ models, $\Pi(x)$  also
obeys an upper bound $|x|^{-3(d-2)}$ \cite{BHH19,Saka07,Saka15}, which raises the
possibility that our results could be extended to these spin models.  However,
for percolation the bound on $\Pi(x)$ is $|x|^{-2(d-2)}$ and for lattice trees and
lattice animals it is $|x|^{-(2d-6)}$
\cite{HHS03}, so for neither does $\Pi$ have finite $(d-2)^{\rm nd}$ moment in any dimension.
Thus our approach cannot apply to percolation, nor to lattice trees and lattice
animals, without a new idea.  A completely different proof of Theorem~\ref{thm:mr} for
high-dimensional percolation is given in \cite{HMS22}; that proof does not apply to
weakly self-avoiding walk.
\end{rk}

The following lemma illustrates the role of a key cancellation due to \refeq{Ezero}.

\begin{lemma}
\label{lem:Ebd}
Let $d>4$ and let $\beta$ be sufficiently small.
Fix $z \in [\frac{1}{\degree},z_c)$ and $m \in [0, \frac 12 m(z)]$.
There is a $c_0>0$ (independent of $z,m,k$)
such that
$|\nabla^\alpha \hat E_z^{(m)}(k)| \le c_0 \beta$
for $|\alpha|  \le d-2$,  and moreover,
\begin{equation}
\lbeq{nabEbd}
    |\nabla^\alpha \hat E_z^{(m)}(k)| \le
    c_0\beta(|k|+m)^{3-|\alpha|}
    \qquad (|\alpha| \le 3).
\end{equation}
\end{lemma}

\begin{proof}
We first prove that $|\nabla^\alpha \hat E_z^{(m)}(k)| \le c_0 \beta$
for $|\alpha| \le d-2$.  To begin, we differentiate the formula for $\hat E_z^{(m)}(k)$
which arises from Fourier transformation of \refeq{Emformula}, and obtain
\begin{align}
\lbeq{nablaEmformula}
    \nabla^\alpha \hat E_z^{(m)}
    &= -(1-\lambda_z)\nabla^\alpha (1-\hat D^{(m)})
    - \lambda_z \hat\Pi_z(0)\nabla^\alpha \hat D^{(m)}
    + \lambda_z \nabla^\alpha \hat \Pi_z^{(m)}.
\end{align}
We have seen previously that $\lambda_z=1+O(\beta)$, and it follows from
Proposition~\ref{prop:Pim} that $\hat\Pi_z(0) = O(\beta)$.
Since $z \in [\frac{1}{\degree},z_c)$ and since $z_c-\frac 1\degree \le O(\beta)$,
the mass $m$ is at most $O(\beta^{1/2})$ so certainly $\nabla^\alpha\hat D^{(m)} (k)=O(1)$.
It also follows from Proposition~\ref{prop:Pim} that $\nabla^\alpha \hat \Pi_z^{(m)}(k)=O(\beta)$ for $|\alpha| \le d-2$.  This proves that, as claimed,
$|\nabla^\alpha \hat E_z^{(m)}(k)| \le c_0 \beta$
for $|\alpha| \le d-2$.

It remains to prove \refeq{nabEbd}.
The case $|\alpha|=3$ has already been proved in the previous paragraph
(of course $3 \le d-2$ for $d \ge 5$),
so we consider now $0 \le |\alpha| \le 2$, beginning with $|\alpha|=0$.
We fix $z$, define $h:[0,\frac 12 m(z)] \times \T^d \to \C$ by
\begin{equation}
    h(m,k) = \hat E_z^{(m)}(k) = \sum_{x\in\Z^d} E_z(x)e^{mx_1}e^{ik\cdot x},
\end{equation}
and expand $h$ to second order in $(m,k)$ with third order error term.
The constant term $h(0,0)$ is zero by \refeq{Ezero}.
First derivatives of $h$ also vanish at $(0,0)$ due to the reflection symmetry
of $E_z(x)$ in each component.  In addition, it follows from \refeq{Ezero} and symmetry that
$\sum_{x\in \Z^d} E_z(x)x_j^2 = 0$ for each $j=1,\ldots, d$, and hence all second derivatives
also vanish at $(0,0)$ (second derivatives other than
$\partial_m^2$, $\partial_m\partial_{k_1}$
and $\partial_{k_j}^2$ all vanish by reflection symmetry).
Therefore, the second order Taylor polynomial of $h$ is identically zero and $h$ is
equal to its third order Taylor remainder, which we consider in its Lagrange form.
All third derivatives of $\hat E_z^{(m)}(k)$, with respect to $(m,k)$, are $O(\beta)$ due
to Proposition~\ref{prop:Pim} (it is only derivatives of $\hat\Pi^{(m)}_z(k)$ that require
careful attention), and thus we see that by increasing the value of $c_0$ if necessary
we obtain
\begin{equation}
    h(m,k) \le c_0\beta(|k|+m)^3 ,
\end{equation}
which is the $|\alpha|=0$ case of \refeq{nabEbd}.

The proof of \refeq{nabEbd} for $|\alpha|=1,2$ follows in exactly the same way, with
the following minor change.  Each $k$-derivative of $h$ brings a factor $x_j$ down
from the exponent.  This reduces the number of vanishing derivatives of $\nabla^\alpha h$ by
$|\alpha|$, with a corresponding reduction in the order of the error in Taylor expansion.
This completes the proof.
\end{proof}

The next lemma provides the relation between $m_0(\mu_z)$ and $m(z)$ claimed in \refeq{mmu}.

\begin{lemma}
\label{lem:mm0}
Let $d>4$ and let $\beta$ be sufficiently small.
If $z \in [\frac{1}{\degree},z_c)$
then $\mu_{z} \in [\frac{1}{2\degree},\frac{1}{\degree})$
and $\frac 12 m(z) \le \frac 23 m_0(\mu_{z})$.
\end{lemma}

\begin{proof}
Let $z \in [\frac{1}{\degree},z_c)$.
As in the proof of Lemma~\ref{lem:Girbd}, $m(z) =O(\beta^{1/2})$.
By \refeq{muzm} and \refeq{F0m2} (with $m=0$),
\begin{align}
\lbeq{mm0a}
    \mu_z \degree & = 1-\lambda_z \hat F_z(0)
     =
    1- zm(z)^2 + O(\beta m(z)^2).
\end{align}
Therefore, $\mu_z=1-O(\beta)$ and hence
$\mu_{z} \in [\frac{1}{2\degree},\frac{1}{\degree})$ for small $\beta$.

Since $\cosh m_0(t)  =  1 + \frac{1-t\degree}{2t}$ by
\refeq{m0def}, we have
\begin{align}
\lbeq{m0t}
    m_0(t)^2 & = \frac{1-t\degree}{t} + O\left(\frac{1-t\degree}{t} \right)^{2}
\end{align}
and hence, by \refeq{mm0a} and the fact that
$\frac{1}{2\degree} \le\mu_z < \frac{1}{\degree} \le z$,
\begin{align}
    m_0(\mu_z)^2 &= \frac{z}{\mu_z} m(z)^2(1+O(\beta))
    \ge    m(z)^2(1+O(\beta)).
\end{align}
For small $\beta$, this gives $\frac 23 m_0(\mu_z) \ge \frac 12 m(z)$, and the proof is complete.
\end{proof}

\begin{proof}[Proof of Proposition~\ref{prop:h}]
Let $|\alpha| \le d-2$, $z \in [\frac{1}{\degree},z_c)$, and $m \le \frac 12 m(z)$.
It follows from the formula for $\hat f^{(m)}_z$ in \refeq{fprod}, together with
the product rule for differentiation, that
$\nabla^\alpha \hat f_z^{(m)}$ involves terms
\begin{equation}
\lbeq{fder}
    (\nabla^{\alpha_1}\hat C_{\mu_{z}}^{(m)})
    (\nabla^{\alpha_2}\hat E_{z}^{(m)}) (\nabla^{\alpha_3}\hat G_{z}^{(m)})
    \quad \text{with} \quad
    |\alpha_1|+|\alpha_2|+|\alpha_3| =|\alpha|.
\end{equation}
The derivatives of $\hat E^{(m)}_z$ are bounded using Lemma~\ref{lem:Ebd}, and
the derivatives of $\hat G^{(m)}_z$ are bounded using Lemma~\ref{lem:Girbd}.
For the derivatives of $\hat C_{\mu_{z}}^{(m)}$, we know from Lemma~\ref{lem:mm0}
that
$\mu_{z} \in [\frac{1}{2\degree},\frac{1}{\degree})$ and
$\frac 23 m_0(\mu_z) \ge \frac 12 m(z) \ge m$,
so we can apply Lemma~\ref{lem:Ck} (with $\sigma = \frac 23$) to obtain
\begin{equation}
    |\nabla^{\alpha_1}\hat C_{\mu_{z}}^{(m)}(k)| \le
    {\rm const}\, \frac{1}{(|k|+m_0(\mu_z))^{2+|\alpha_1|}}
    \le
    {\rm const}\, \frac{1}{(|k|+m)^{2+|\alpha_1|}}.
\end{equation}
Altogether, these facts lead to an upper bound for \refeq{fder} of order
\begin{equation}
    \frac{\beta(|k|+m)^{3-\min\{|\alpha_2|,3\}}}{(|k|+m)^{2+ |\alpha_1|} (|k|+m)^{2+|\alpha_3|}}
    =
    \frac{\beta}{(|k|+m)^{1+ |\alpha_1|+\min\{|\alpha_2|,3\} +|\alpha_3|}}
    \le
    \frac{\beta}{|k|^{d-1}}.
\end{equation}
The integral of the right-hand side over $\T^d$ is $O(\beta)$, so the proof is complete.
\end{proof}

\section{The plateau for the torus two-point functions}
\label{sec:plateaupf}

The proofs of Theorems~\ref{thm:plateau} and \ref{thm:plateau-srw} are largely
the same so we present them
together.  We separate the proofs of the upper and lower bounds, beginning with
the upper bound.
We write $\Gamma_z(x)$ to denote either of $C_z(x)$ or $G_z(x)$ in discussions
that apply to both.

As a preliminary, we observe that
if $x \in \T_r^d$ is regarded as a point in $[-\frac r2, \frac r2)^d \cap \Z^d$
then $\|x+ru\|_\infty \asymp r\|u\|_\infty$ uniformly in nonzero $u\in \Z^d$, since
\begin{equation}
\lbeq{xulb}
    \|x+ru\|_\infty \ge \|ru\|_\infty - \frac r2 \ge  \|ru\|_\infty - \frac 12\|ru\|_\infty
    = \frac 12 \|ru\|_\infty
    ,
\end{equation}
\begin{equation}
\lbeq{xuub}
    \|x+ru\|_\infty \le  \frac r2 + \|ru\|_\infty  \le  \frac 12\|ru\|_\infty + \|ru\|_\infty
    = \frac 32 \|ru\|_\infty.
\end{equation}

\subsection{Upper bound for the torus two-point functions}
\label{sec:torusub}

As discussed below \refeq{CN},
walks on the torus are in a one-to-one correspondence with
walks on $\Z^d$ via unfolding,
and the simple random walk torus two-point function is given by
\begin{equation}
\lbeq{CTZ}
    C^{\T}_z(x) = C_z(x) + \sum_{u\in \Z^d: u \neq 0} C_z(x+ru).
\end{equation}
 For weakly self-avoiding walk the equality does not hold, because the unfolding of a
torus walk can have fewer intersections than the folded walk.  Thus the above is
replaced by an inequality
\begin{equation}
\lbeq{GGG}
    G^{\T}_z(x) \le G_z(x) + \sum_{u\in \Z^d: u \neq 0} G_z(x+ru).
\end{equation}
By Theorem~\ref{thm:mr} and Proposition~\ref{prop:srw}
(recall that $\Gamma_z$ denotes either $C_z$ or $G_z$),
\begin{align}
\lbeq{Gamma2}
    \Gamma_z(x) & \le c \frac{1}{1\vee |x|^{d-2}} e^{-c'\nu(z) |x|} \qquad (z \in (0, z_*)),
\end{align}
with $\nu(z)$ equal to $m_0(z)$ or $m(z)$, and with $z_*$ equal to $\frac{1}{\degree}$
or $z_c$, for $C_z$ and $G_z$ respectively.
The plateau upper bounds then follow immediately, as follows.

\begin{proof}[Proof of upper bounds in Theorems~\ref{thm:plateau} and \ref{thm:plateau-srw}]
Let $z\in (0,z_*)$,
let $d>2$ for simple random walk, and let $d>4$ and $\beta$ be small
for weakly self-avoiding walk.
By \refeq{Gamma2} and \refeq{xulb},
\begin{align}
\lbeq{plateau-ub1}
    \sum_{u\in \Z^d: u \neq 0} \Gamma_z(x+ru)
    & \le c \sum_{u\neq 0}    \frac{1}{|x+ru|^{d-2}} e^{-c' \nu(z) |x+ru|}
    \le c_0 \sum_{u\neq 0}    \frac{1}{|ru|^{d-2}} e^{-c_0' \nu(z) |ru|}.
\end{align}
We bound the sum on the right-hand side by an integral and make the change of
variables $y=\nu ru$ to obtain
\begin{align}
\lbeq{plateau-ub2}
    \sum_{u\in \Z^d: u \neq 0} \Gamma_z(x+ru)
    & \le c_1 \frac{1}{r^d\nu(z)^2} \int_{\R^d} du  \frac{1}{|y|^{d-2}} e^{-c_0' |y|}
    \le
    c_1' \frac{1}{r^d\nu(z)^2} .
\end{align}
It remains to show that $\nu(z)^{-2} \le {\rm const}\,\chi(z)$.
Fix any $z_1 \in (0,z_*)$.  For $z \le z_1$, since $\nu$ is decreasing and
since $1=\chi(0) \le \chi(z)$, we have $\nu(z)^{-2} \le \nu(z_1)^{-2} \le
\nu(z_1)^{-2}\chi(z)$ and the desired upper bound ${\rm const}\,r^{-d}\chi(z)$
follows for $z \in (0,z_1]$.
We can choose $z_1$ close enough to $z_c$ that $\nu(z)^{-2}$ and $\chi(z)$ are
comparable for $z\in (z_1,z_*)$, since both are asymptotic to $(1-z/z_*)^{-1}$,
and this gives the desired estimate for $z\in (z_1,z_*)$ and thus completes the proof.\footnote{The mechanism in this proof applies more generally---it is not necessary
that the decay be exponential as in \refeq{Gamma2}.
For example, for a random walk on $\Z^d$ whose step distribution is given by a fractional power
$-(-\Delta)^{\alpha/2}$
of the discrete Laplacian (a step from $0$ to $x$ has probability
$-(-\Delta)^{\alpha/2}_{0x} \asymp |x|^{-(d+\alpha)}$),
the bound \refeq{Gamma2} is replaced, for $d \ge 1$, $\alpha \in (0,\min\{2, d\})$,
 $m^2 \in [0,1]$, and $x \neq 0$, by
(see \cite[Section~2.1]{Slad18})
\[
        ((-\Delta)^{\alpha/2}+m^2)^{-1}_{0x}
        \le
        c \frac{1}{|x|^{d-\alpha}} \frac{1}{1+m^4|x|^{2\alpha}}.
\]
The steps in \refeq{plateau-ub1}--\refeq{plateau-ub2} now
give an upper bound $r^{-d}m^{-2}$, which equals $r^{-d}\chi$ since
$\sum_{x\in\Z^d} ((-\Delta)^{\alpha/2}+m^2)^{-1}_{0x}=m^{-2}$.  }
\end{proof}

\subsection{Lower bound for the torus two-point functions}

We first consider dimensions $d>4$, which is the case relevant for the weakly self-avoiding walk.
Let
\begin{equation}
    \rho=z_*-z.
\end{equation}

\begin{lemma}
\label{lem:Gamma}
Let $d>4$, and for weakly self-avoiding walk let $\beta$ be sufficiently small.
There are $a_i>0$ such that, for all $x \in \Z^d$ and all $z \in [\frac 12 z_*,z_*]$,
\begin{align}
\lbeq{Gamma0}
    a_4 \frac{1}{1\vee |x|^{d-2}} \le \Gamma_{z_*}(x) & \le a_5 \frac{1}{1\vee |x|^{d-2}},
    \\
\lbeq{Gamma4}
    \Gamma_{z_*}(x)-\Gamma_z(x) & \le a_3 \rho \frac{1}{1\vee |x|^{d-4}}.
\end{align}
\end{lemma}

\begin{proof}
The fact that \refeq{Gamma0} holds for simple random walk is just the standard decay
for the critical two-point function \cite{LL10}.
For weakly self-avoiding walk \refeq{Gamma0} follows from
\refeq{Gzc}.

For \refeq{Gamma4}, we first claim that
\begin{equation}
\lbeq{convest0}
    z\frac{d}{dz}\Gamma_{z}(x) \le  (\Gamma_{z}*\Gamma_{z})(x) .
\end{equation}
To prove this for simple random walk, we use the definition of $C_z(x)$ in \refeq{Czdef}
to see that
\begin{equation}
        z\frac{d}{dz}C_{z}(x) = \sum_{n=1}^\infty \sum_{\omega\in \Wcal_n(x)} nz^n
        \le
         \sum_{n=0}^\infty \sum_{\omega\in \Wcal_n(x)} \sum_{m=0}^n z^{n-m}z^m,
\end{equation}
and then divide the walk $\omega$ into subwalks of lengths $m$ and $n-m$ to factor
the right-hand side as a convolution.  The same technique applies to weakly self-avoiding
walk, with the additional step of using the inequality
\begin{equation}
    \prod_{0 \le s<t \le n}(1+\beta U_{st})
    \le
    \prod_{0 \le s_1<t_1 \le m}(1+\beta U_{s_1t_1})
    \prod_{m \le s_2<t_2 \le n}(1+\beta U_{s_2t_2})
\end{equation}
to bound the interaction for $\omega$ by
a product of interactions for the two subwalks.
This proves \refeq{convest0}, and then monotonicity in $z$,
the upper bound of \refeq{Gamma0},
and the elementary convolution estimate\footnote{The estimate implies that the
convolution with itself of a function bounded by $|x|^{-(d-2)}$ is bounded by
$|x|^{-(d-4)}$ if $d>4$.}
\cite[Proposition~1.7(i)]{HHS03} give
\begin{equation}
\lbeq{convest}
    z\frac{d}{dz}\Gamma_{z}(x)
    \le  (\Gamma_{z_*}*\Gamma_{z_*})(x)
    \le   c_0\frac{1}{1\vee |x|^{d-4}}.
\end{equation}
Integration of \refeq{convest} (using the assumed lower bound on the factor $z$)
gives \refeq{Gamma4}, and the proof is complete.
\end{proof}

\begin{rk}
\label{rk:Gamma}
We can now prove the claim in
\refeq{Gzzc}. The upper bound follows by bounding $\Gamma_z(x)$ by
$\Gamma_{z_*}(x)$ and using the upper bound of \refeq{Gamma0}.
For the lower bound of \refeq{Gzzc}, we take $\rho \le \epsilon r^{-2}$ with
$\epsilon=\frac 2d a_3^{-1}a_4$
and apply Lemma~\ref{lem:Gamma} to see that
\begin{align}
    \Gamma_z(x) & = \Gamma_{z_*}(x) - (\Gamma_{z_*}(x)-\Gamma_z(x))
    \nnb & \ge
    a_4\frac{1}{1\vee |x|^{d-2}} - a_3\epsilon \frac{1}{r^2} \frac{1}{1\vee |x|^{d-4}}
    \ge \frac 12 a_4 \frac{1}{1\vee |x|^{d-2}},
\end{align}
where we used $|x|^2 \le dr^{2}/4$ in the last step.
The constant $c_3$ in \refeq{plateau}--\refeq{Gzzc} can be taken to be at most $\epsilon$.
\end{rk}

For the lower bound in dimensions $d>4$, we first consider the easier case of
simple random walk.  Its proof is used also for weakly self-avoiding walk.

\begin{proof}[Proof of lower bound in Theorem~\ref{thm:plateau-srw} for $d>4$]
Let $d>4$ and $x\in \T_r^d$,
and suppose that $\rho\le c_3' r^{-2}$, with the constant $c_3'$ to be
chosen in the proof.
Since $\chi_0(z) = (1-z\degree)^{-1} = (\degree \rho)^{-1}$,
by \refeq{CTZ} it suffices to prove
for some $c>0$ that
\begin{equation}
\lbeq{Csum}
    \sum_{u\in \Z^d: u \neq 0} C_z(x+ru) \ge
    \frac{c}{r^d \rho}.
\end{equation}
For a lower bound,
we only sum over $|u|\le L$ with (large) $L$ depending on $r,\rho$
to be chosen later in the proof.
By \refeq{Gamma0}--\refeq{Gamma4}, for $y \neq 0$,
\begin{align}
    C_{z }(y) & = C_{z_*}(y) - \big(C_{z_*}(y) - C_{z }(y)\big)
    \ge
    a_4\frac{1}{|y|^{d-2}} - a_3\rho \frac{1}{|y|^{d-4}}  .
\end{align}
Therefore, by \refeq{xulb}--\refeq{xuub},
\begin{align}
    \sum_{u\in \Z^d: u \neq 0} C_z(x+ru) &\ge
    \sum_{0<|u| \le L} C_{z}(x+ru)
    \nnb & \ge
    c_0\sum_{0<|u|\le L} \frac{1}{|ru|^{d-2}} - c_0' \rho\sum_{0<|u|\le L}\frac{1}{|ru|^{d-4}}
    \nnb &\ge
    c_1\frac{1}{r^{d-2}}L^2 - c_1' \rho  \frac{1}{r^{d-4}} L^4
    \nnb & =
    c_1\frac{1}{r^{d-2}}L^2 \left(1 - c_2 \rho  r^2 L^2 \right).
\end{align}
We choose $L^2=(2c_2 \rho  r^{2})^{-1}$, and then $L^2  \ge  (2c_2c_3')^{-1}$
is large if we take $c_3'$ to be small.  This choice gives
\begin{align}
    \sum_{u\in \Z^d: u \neq 0} C_z(x+ru)
    & \ge
    \frac 12 c_1 \frac{L^2}{r^{d-2}} =
    \frac{c_1}{4c_2}
    \frac{1}{r^d \rho}  ,
\end{align}
which proves \refeq{Csum} and completes the proof.
\end{proof}

\begin{proof}[Proof of lower bound of Theorem~\ref{thm:plateau}]
We consider dimensions $d>4$ and small $\beta$, and $z$ such
that $\rho=z_c-z$ obeys $c_4\beta^{1/2}r^{-d/2} \le \rho \le c_3 r^{-2}$ with
$c_3$ equal to $c_3'$ from the previous proof and with $c_4$ to be
chosen at the end of the proof.
We seek a lower bound of the form $r^{-d}\chi$ for the difference
\begin{equation}
    \psi_{r,z}^{\T}(x) =
    G^{\T}_{z}(x) - G_{z}(x)
    \qquad (x \in \T^d)
    .
\end{equation}

A torus walk to $x$ unfolds to a walk on $\Z^d$ ending at $x$ or at a point $x+ru$
with $u$ a nonzero point in $\Z^d$.  The weight associated to the unfolded
walk, as a weakly self-avoiding walk on $\Z^d$, can be larger than the weight of
the original torus walk which is penalised by visits of its unfolding to distinct points
in $\Z^d$ with the same projection to the torus.  Without this penalty, the unfolded
walks would have weight
\begin{equation}
    \psi_{r,z}(x) = \sum_{u \in \Z^d: u \neq 0}G_z(x+ru)
    \qquad (x \in \T_r^d).
\end{equation}
Exactly as in the proof of the lower bound of Theorem~\ref{thm:plateau-srw},
$\psi_{r,z}(x)$  is bounded below by a multiple of $r^{-d}\chi$
if $\rho \ge c_3  r^{-2}$, since that proof only
used \refeq{Gamma0}--\refeq{Gamma4}.
For later reference, note that in the proof of the upper bound of Theorem~\ref{thm:plateau}  we in fact obtained an upper bound on $\psi_{r,z}(x)$ without using this particular
notation, and hence
(provided $\rho \le c_3 r^{-2}$ for the lower bound),
\begin{equation}
\lbeq{psiub}
    c r^{-d}\chi(z) \le \psi_{r,z}(x) \le c' r^{-d}\chi(z)
    \qquad (x \in \T_r^d).
\end{equation}
We make the decomposition
\begin{equation}
    \psi_{r,z}^{\T}(x) =  \psi_{r,z}(x) - (\psi_{r,z}(x) -\psi^{\T}_{r,z}(x) ).
\end{equation}
By the lower bound of \refeq{psiub},
it suffices to show that the subtracted term, which by
\refeq{GGG} is nonnegative, obeys
\begin{equation}
\lbeq{psigoal}
    \psi_{r,z}(x) -\psi^{\T}_{r,z}(x)
    \le
    \frac 12 c r^{-d}\chi(z)
    \qquad (x \in \T_r^d),
\end{equation}
with $c$ the constant of \refeq{psiub}.

Let $\pi_r : \Z^d \to \T_r^d$ be the canonical projection map onto the torus
$\T_r^d = (\Z/r\Z)^d$.
In the following, all walks are on $\Z^d$.
Given an $n$-step walk $\omega$ and $0 \le s < t \le n$, we define
\begin{align}
    U^{\T}_{st}(\omega) & =
    \begin{cases}
        -1 & (\pi_r \omega(s) = \pi_r \omega(t))
        \\
        0 & (\text{otherwise}),
    \end{cases}
    \\
    U_{st}(\omega) & =
    \begin{cases}
        -1 & (\omega(s) = \omega(t))
        \\
        0 & (\text{otherwise}),
    \end{cases}
    \\
    U^+_{st}(\omega) & =
    \begin{cases}
        -1 & (\pi_r\omega(s) = \pi_r\omega(t) \; \text{and}\; \omega(s)\neq \omega(t))
        \\
        0 & (\text{otherwise}),
    \end{cases}
\end{align}
as well as
\begin{align}
    K^{\T}(\omega) & = \prod_{0 \le s<t \le n} (1+\beta  U^{\T}_{st}(\omega)),
    \\
    K(\omega) & = \prod_{0 \le s<t \le n} (1+\beta  U_{st}(\omega)),
    \\
    K^{+}(\omega) & = \prod_{0 \le s<t \le n} (1+\beta  U^{+}_{st}(\omega)).
\end{align}
Note that $K^{\T}(\omega)=K(\omega)K^{+}(\omega)$.
By definition,
\begin{align}
    \psi_{r,z}(x) & =
    \sum_{u \neq 0} \sum_{n=0}^\infty \sum_{\omega \in \Wcal_n(x+ru)} z^{n} K(\omega)
    \qquad (x \in \T_r^d),
    \\
    \psi^{\T}_{r,z}(x) & =
    \sum_{u \neq 0} \sum_{n=0}^\infty \sum_{\omega \in \Wcal_n(x+ru)} z^{n} K^{\T}(\omega)
    \qquad (x \in \T_r^d),
\end{align}
and hence
\begin{equation}
\lbeq{psidif}
    \psi_{r,z}(x) -\psi^{\T}_{r,z}(x)  =
    \sum_{u \neq 0} \sum_{n=0}^\infty \sum_{\omega \in \Wcal_n(x+ru)} z^{n}
    K (\omega) [1-K^+(\omega)].
\end{equation}

The inequality
\begin{equation}
    1- \prod_{a\in A}(1-u_a)
    \le \sum_{a \in A}u_a
    \qquad (u_a \in [0,1])
\end{equation}
follows by induction on the
cardinality of the set $A$.  We apply it to $1-K^+$ with $u_a=\beta|U_{st}^+(\omega)|$.
The result is
\begin{align}
    \psi_{r,z}(x) -\psi^{\T}_{r,z}(x)  &\le
    \beta
    \sum_{u \neq 0} \sum_{n=0}^\infty \sum_{0 \le s<t \le n} z^{n}
    \sum_{\omega \in \Wcal_n(x+ru)}
    K(\omega)  |U^+_{st}(\omega)|
    .
\lbeq{inex}
\end{align}
The factor $U^+_{st}(\omega)$ is zero unless $\omega$ visits distinct points at times $s,t$
that
project to the same torus point.  We call two such points $y$ and $y+rv$ with $v\neq 0$, which entails
that $\omega$ travel from $0$ to $y$, from $y$ to $y+rv$, and from $y+rv$ to $x+ru$.
We can therefore decompose $\omega$ as the concatenation of
three subwalks, respectively $\omega_1 \in \Wcal_{n_1}(y)$,
$\omega_2\in \Wcal_{n_2}(rv)$, and $\omega_3 \in \Wcal_{n_3}(x-y + r(u-v))$,
of lengths $n_1=s$, $n_2=t-s$, and $n_3=n-t$.
Also $z^n=z^{n_1}z^{n_2}z^{n_3}$.
The sums over $n,s,t$ in \refeq{inex} are
equivalent to summing over $n_1 \ge 0$, $n_2 \ge 1$, $n_3 \ge 0$, and $K(\omega)$
is bounded above by the product $K(\omega_1)K(\omega_2)K(\omega_3)$.  After using this
last bound (and relaxing the condition that $u$ be nonzero),
\refeq{inex} is bounded above by three independent sums over walks, which
leads to an upper bound
\begin{align}
\lbeq{3walks}
     &\beta \sum_{y\in\Z^d}G_z(y)
     \sum_{v\in \Z^d, v\neq 0}  G_z(rv)
     \sum_{u \in \Z^d}G_z(x-y +r(u-v))
     \nnb & \quad =
     \beta \psi_{r,z}(0)
     \sum_{y,w\in\Z^d}G_z(y) G_z(x-y +rw)
     .
\end{align}

To estimate the right-hand side of \refeq{3walks} for $x \in \T_r^d$, we proceed as follows.
The constant $C$ in this paragraph can change value from one occurrence to the next.
For fixed $w$, by \refeq{Gmr} and by $|y-(x+rw)| \ge |x+rw|-|y|$ we have
\begin{align}
    \sum_{y\in\Z^d} G_z(y) G_z(x-y +rw)
    & \le C \sum_{y\in\Z^d}
    \frac{1}{1\vee |y|^{d-2}} e^{-c_1 m(z) |y|}
    \frac{1}{1\vee |y-(x+rw)|^{d-2}} e^{-c_1 m(z) |y-(x+rw)|}
    \nnb & \quad
    \le C e^{-c_1 m(z) |x+rw|} \sum_{y\in\Z^d}
    \frac{1}{1\vee |y|^{d-2}}
    \frac{1}{1\vee |y-(x+rw)|^{d-2}}
    \nnb & \quad
    \le C e^{-c_1 m(z) |x+rw|}
    \frac{1}{1\vee |x+rw|^{d-4}},
\end{align}
with the last inequality a consequence of the convolution estimate \cite[Proposition~1.7(i)]{HHS03}.
The $w=0$ term is bounded by a constant.  For the rest of the sum over $w$,
as in the proof of the upper bound in Section~\ref{sec:torusub}, we use
\begin{align}
    \sum_{w \neq 0} \frac{1}{1\vee |x+rw|^{d-4}} e^{-c_1 m(z) |x+rw|}
    & \le C
    \sum_{w \neq 0} \frac{1}{|rw|^{d-4}} e^{-c_1' m(z) |rw|}
    \nnb & \le C
    \int_{\R^d} \frac{1}{|ra|^{d-4}} e^{-c_1' m(z) |ra|}da
    \nnb & = C
    \frac{1}{m^4 r^d}
    \le C
    \frac{\chi(z)^2}{r^d}.
\end{align}
With the $w=0$ term and the upper bound of \refeq{psiub}, this leads to
\begin{align}
     \psi_{r,z}(x) -\psi^{\T}_{r,z}(x) &  \le
     {\rm const}\, \frac{\chi(z)}{r^d}
     \left(\beta
     +
     \beta\frac{\chi(z)^2}{r^d} \right)
     .
\end{align}
This proves \refeq{psigoal} provided that $\beta$ and
$\beta r^{-d}\chi(z)^2$ are sufficiently small. Since $\chi(z) \asymp \rho^{-1}$,
it is sufficient if
$\beta r^{-d} \rho^{-2}$ is sufficiently small,
i.e., if $\rho^2 \ge c_4^2 \beta r^{-d}$ for some large $c_4$.
We have assumed this last inequality as a hypothesis for this reason.
This concludes the proof of \refeq{psigoal} and therefore completes
the proof.
\end{proof}

Finally, we prove the lower bound in the remaining dimensions $d=3,4$ for simple random walk.

\begin{proof}[Proof of lower bound of Theorem~\ref{thm:plateau-srw} for $d=3,4$]
As in the proof for $d>4$, we seek a lower bound on
$\sum_{u \neq 0} C_z(x+ru)$.  We again write $z_*=\frac{1}{\degree}$ and $\rho=z_*-z$.
The lower bound is claimed to hold for $\rho \le O(r^{-2})$, so we are considering
small $\rho$ here.

We first prove that
\begin{align}
\lbeq{Gamma1}
    C_{z_*}(x)-C_z(x) & \le {\rm const}  \times
    \begin{cases}
        \rho^{1/2} & (d=3)
        \\
        \rho |\log \rho| & (d=4).
    \end{cases}
\end{align}
By \refeq{CAdelta},
the left-hand side of \refeq{Gamma1} can be written as the Fourier integral
\begin{align}
    C_{z_*}(x)-C_z(x)
    & =
    \int_{\T^d} e^{-ik\cdot x} \left( \frac{1}{1-\hat D(k)} -\frac{1}{1-z\degree\hat D(k)}\right) \frac{dk}{(2\pi)^d}
    \nnb & =
    \degree\rho
    \int_{\T^d} e^{-ik\cdot x} \left( \frac{\hat D(k)}{[1-\hat D(k)] [1-z\degree\hat D(k)]}\right) \frac{dk}{(2\pi)^d}
    .
\lbeq{Cdif}
\end{align}
In the denominator, $1-\hat D(k)$ is bounded below by a multiple of $|k|^2$, while
$1-z\degree \hat D(k) = \degree \rho + z\degree[1-\hat D(k)]$ is bounded below
by a multiple of $\rho + |k|^2$ (since $z\degree$ is bounded below because we are
considering small $\rho$).
The right-hand side of \refeq{Cdif} is therefore bounded above by a multiple of
\begin{align}
    &  \rho \int_{\T^d} \frac{1}{|k|^2 (\rho + |k|^2)} dk
    =
     \rho^{(d-2)/2} \int_{\rho^{-1/2}\T^d} \frac{1}{|l|^2 (1 + |l|^2)} dl.
\end{align}
For $d=3$ this last integral is bounded as $\rho\to 0$,
while for $d=4$ it is $O(|\log \rho|)$.
This proves \refeq{Gamma1}.

For simple random walk, \refeq{Gamma0} continues to hold in dimensions $d=3,4$ \cite{LL10}.
For $d=3$ and $y \neq 0$, by \refeq{Gamma1} we therefore have
\begin{align}
    C_z(y) & = C_{z_*}(y) - (C_{z_*}(y) - C_{z}(y))
    \ge
    a_4\frac{1}{|y|} - c_1 \rho^{1/2},
\end{align}
so, for large $L$, by \refeq{xuub},
\begin{align}
    \sum_{u\in \Z^d: u \neq 0} C_z(x+ru) & \ge
    \sum_{0<|u| \le L} C_{z}(x+ru)
    \ge
    c_0'\sum_{0<|u|\le L} \frac{1}{|ru|} - c_1' \rho^{1/2}L^3
    \nnb & \ge
    c_0'' \frac{1}{r} L^2 - c_1' \rho^{1/2}L^3
    =
    \frac 12 c_0'' \frac{1}{r^3\rho},
\end{align}
where we took $L= c_0''(2c_1' \rho^{1/2} r)^{-1}$ in the last step.
This $L$ is large provided that our assumption $\rho \le c_3'r^{-2}$ holds with
$c_3'$ chosen sufficiently small.  By \refeq{CTZ}, this proves the desired lower
bound \refeq{plateau-srw} for $d=3$.

Finally, for $d=4$, a similar computation gives
\begin{align}
    C_z(y) & \ge
    a_4 \frac{1}{|y|^2} - c_1\rho   | \log\rho  |,
\end{align}
so
\begin{align}
    \sum_{u\in \Z^d: u \neq 0} C_z(x+ru) & \ge
    c_0' \frac{1}{r^2} L^2 - c_1' \rho |\log\rho| L^4
    = c_0' \frac{1}{r^2} L^2 (1- c_1'' \rho |\log \rho| r^2 L^2).
\end{align}
Now we choose
$L^2 = (2c_1''\rho|\log \rho| r^2)^{-1}$
to obtain
\begin{equation}
    \sum_{u\in \Z^d: u \neq 0} C_z(x+ru) \ge c_0'' \frac{1}{r^4 \rho |\log\rho|}.
\end{equation}
The above assumes that $L$ is large, which is true by our assumption that
$\rho|\log \rho| r^2$ is sufficiently small.
This completes the proof.
\end{proof}

\section*{Acknowledgements}

This work was supported in part by NSERC of Canada.
I am grateful to
Emmanuel Michta for many discussions and suggestions, to an anonymous referee
for identifying an error in a previous version of Lemma~\ref{lem:diagram-m}, to
Tyler Helmuth for helpful discussions, and to Timothy Garoni for
providing a preliminary version of \cite{ZGDG20}.


\begin{thebibliography}{10}

\bibitem{Barl17}
M.T. Barlow.
\newblock {\em Random Walks and Heat Kernels on Graphs}.
\newblock Cambridge University Press, Cambridge, (2017).

\bibitem{BBS-saw4}
R.~Bauerschmidt, D.C. Brydges, and G.~Slade.
\newblock Critical two-point function of the 4-dimensional weakly self-avoiding
  walk.
\newblock {\em Commun.\ Math.\ Phys.}, {\bf 338}:169--193, (2015).

\bibitem{BHK18}
E.~Bolthausen, R.~van~der Hofstad, and G.~Kozma.
\newblock Lace expansion for dummies.
\newblock {\em Ann.\ I.\ Henri Poincar\'{e} Probab.\ Statist.}, {\bf
  54}:141--153, (2018).

\bibitem{BCHSS05a}
C.~Borgs, J.T. Chayes, R.~van~der Hofstad, G.~Slade, and J.~Spencer.
\newblock Random subgraphs of finite graphs: {I}. {The} scaling window under
  the triangle condition.
\newblock {\em Random Struct. Alg.}, {\bf 27}:137--184, (2005).

\bibitem{BCHSS05b}
C.~Borgs, J.T. Chayes, R.~van~der Hofstad, G.~Slade, and J.~Spencer.
\newblock Random subgraphs of finite graphs: {II}. {The} lace expansion and the
  triangle condition.
\newblock {\em Ann. Probab.}, {\bf 33}:1886--1944, (2005).

\bibitem{BHH19}
D.C. Brydges, T.~Helmuth, and M.~Holmes.
\newblock The continuous-time lace expansion.
\newblock {\em Commun. Pure Appl. Math.}, {\bf 74}:2251--2309, (2021).

\bibitem{BS85}
D.C. Brydges and T.~Spencer.
\newblock Self-avoiding walk in 5 or more dimensions.
\newblock {\em Commun. Math. Phys.}, {\bf 97}:125--148, (1985).

\bibitem{CJN21}
F.~Camia, J.~Jiang, and C.M. Newman.
\newblock The effect of free boundary conditions on the {Ising} model in high
  dimensions.
\newblock {\em Probab. Theory Related Fields}, {\bf 181}:311--328, (2021).

\bibitem{CH20}
S.~Chatterjee and J.~Hanson.
\newblock Restricted percolation critical exponents in high dimensions.
\newblock {\em Commun. Pure Appl. Math.}, {\bf 73}:2370--2429, (2020).

\bibitem{CC86b}
J.T. Chayes and L.~Chayes.
\newblock {Ornstein-Zernike} behavior for self-avoiding walks at all
  noncritical temperatures.
\newblock {\em Commun. Math. Phys.}, {\bf 105}:221--238, (1986).

\bibitem{DGGNZ19}
Y.~Deng, T.M. Garoni, J.~Grimm, A.~Nasrawi, and Z.~Zhou.
\newblock The length of self-avoiding walks on the complete graph.
\newblock {\em J. Stat. Mech: Theory Exp.}, 103206, (2019).

\bibitem{Graf14}
L.~Grafakos.
\newblock {\em Classical Fourier Analysis}.
\newblock Springer, New York, 3rd edition, (2014).

\bibitem{Grim18}
J.~Grimm.
\newblock {\em Finite-size effects in high dimensional physical systems}.
\newblock PhD thesis, Monash University, (2018).
\newblock \url{https://doi.org/10.26180/5b9978f9dd0b6}.

\bibitem{Grim99}
G.~Grimmett.
\newblock {\em Percolation}.
\newblock Springer, Berlin, 2nd edition, (1999).

\bibitem{Hara90}
T.~Hara.
\newblock Mean field critical behaviour for correlation length for percolation
  in high dimensions.
\newblock {\em Probab. Theory Related Fields}, {\bf 86}:337--385, (1990).

\bibitem{Hara08}
T.~Hara.
\newblock Decay of correlations in nearest-neighbor self-avoiding walk,
  percolation, lattice trees and animals.
\newblock {\em Ann. Probab.}, {\bf 36}:530--593, (2008).

\bibitem{HHS03}
T.~Hara, R.~van~der Hofstad, and G.~Slade.
\newblock Critical two-point functions and the lace expansion for spread-out
  high-dimensional percolation and related models.
\newblock {\em Ann. Probab.}, {\bf 31}:349--408, (2003).

\bibitem{HS92a}
T.~Hara and G.~Slade.
\newblock Self-avoiding walk in five or more dimensions. {I.} {The} critical
  behaviour.
\newblock {\em Commun.\ Math.\ Phys.}, {\bf 147}:101--136, (1992).

\bibitem{HH17book}
M.~Heydenreich and R.~van~der Hofstad.
\newblock {\em Progress in High-Dimensional Percolation and Random Graphs}.
\newblock Springer International Publishing Switzerland, (2017).

\bibitem{HS14}
R.~van~der Hofstad and A.~Sapozhnikov.
\newblock Cycle structure of percolation on high-dimensional tori.
\newblock {\em Ann.\ Inst.\ H.\ Poincar\'e Probab.\ Statist.}, {\bf
  50}:999--1027, (2014).

\bibitem{HMS22}
T.~Hutchcroft, E.~Michta, and G.~Slade.
\newblock High-dimensional near-critical percolation and the torus plateau.
\newblock Preprint, \url{https://arxiv.org/pdf/2107.12971}, (2021).

\bibitem{ID89a}
C.~Itzykson and J-M. Drouffe.
\newblock {\em Statistical Field Theory}, volume~{I}.
\newblock Cambridge University Press, Cambridge, (1989).

\bibitem{LL10}
G.F. Lawler and V.~Limic.
\newblock {\em Random Walk: A Modern Introduction}.
\newblock Cambridge University Press, Cambridge, (2010).

\bibitem{LP17}
D.A. Levin and Y.~Peres.
\newblock {\em Markov Chains and Mixing Times}.
\newblock American Mathematical Society, Providence, RI, 2nd edition, (2017).

\bibitem{LM16}
P.H. Lundow and K.~Markstr\"om.
\newblock The scaling window of the $5${D} {Ising} model with free boundary
  conditions.
\newblock {\em Nucl. Phys. B}, {\bf 911}:163--172, (2016).

\bibitem{MS93}
N.~Madras and G.~Slade.
\newblock {\em The Self-Avoiding Walk}.
\newblock Birkh{\"a}user, Boston, (1993).

\bibitem{Mich22}
E.~Michta.
\newblock The scaling limit of the weakly self-avoiding walk on a
  high-dimensional torus.
\newblock Preprint, \url{https://arxiv.org/pdf/2203.07695}, (2022).

\bibitem{MPS23}
E.~Michta, J.~Park, and G.~Slade.
\newblock Universal finite-size scaling for the $4$-dimensional multi-component
  hierarchical $|\varphi|^4$ model.
\newblock In preparation.

\bibitem{MS22}
E.~Michta and G.~Slade.
\newblock Weakly self-avoiding walk on a high-dimensional torus.
\newblock Preprint, \url{https://arxiv.org/pdf/2107.14170}, (2021).

\bibitem{MS21}
E.~Michta and G.~Slade.
\newblock Asymptotic behaviour of the lattice {Green} function.
\newblock {\em ALEA, Lat. Am. J. Probab. Math. Stat.}, {\bf 19}:957--981,
  (2022).

\bibitem{Papa06}
V.~Papathanakos.
\newblock {\em Finite-Size Effects in High-Dimensional Statistical Mechanical
  Systems: {The} {Ising} Model with Periodic Boundary Conditions}.
\newblock PhD thesis, Princeton University, (2006).

\bibitem{Saka07}
A.~Sakai.
\newblock Lace expansion for the {Ising} model.
\newblock {\em Commun. Math. Phys.}, {\bf 272}:283--344, (2007).
\newblock Correction: A.~Sakai. Correct bounds on the Ising lace-expansion
  coefficients. {\it Commun. Math. Phys.}, {\bf 392}:783--823, (2022).

\bibitem{Saka15}
A.~Sakai.
\newblock Application of the lace expansion to the $\varphi^4$ model.
\newblock {\em Commun. Math. Phys.}, {\bf 336}:619--648, (2015).

\bibitem{Slad06}
G.~Slade.
\newblock {\em The Lace Expansion and its Applications.}
\newblock Springer, Berlin, (2006).
\newblock Lecture Notes in Mathematics Vol. 1879. Ecole d'Et\'{e} de
  Probabilit\'{e}s de Saint--Flour XXXIV--2004.

\bibitem{Slad18}
G.~Slade.
\newblock Critical exponents for long-range {$O(n)$} models below the upper
  critical dimension.
\newblock {\em Commun. Math. Phys.}, {\bf 358}:343--436, (2018).

\bibitem{Slad20}
G.~Slade.
\newblock Self-avoiding walk on the complete graph.
\newblock {\em J. Math. Soc. Japan}, {\bf 72}:1189--1200, (2020).

\bibitem{Slad20_lace}
G.~Slade.
\newblock A simple convergence proof for the lace expansion.
\newblock {\em Ann.\ I.\ Henri Poincar\'{e} Probab.\ Statist.}, {\bf
  58}:26--33, (2022).

\bibitem{WY14}
M.~Wittmann and A.P. Young.
\newblock Finite-size scaling above the upper critical dimension.
\newblock {\em Phys. Rev. E}, {\bf 90}:062137, (2014).

\bibitem{ZGDG20}
Z.~Zhou, J.~Grimm, Y.~Deng, and T.M. Garoni.
\newblock Random-length random walks and finite-size scaling on
  high-dimensional hypercubic lattices {I}: {Periodic} boundary conditions.
\newblock Preprint, \url{https://arxiv.org/pdf/2008.00913}, (2020).

\bibitem{ZGFDG18}
Z.~Zhou, J.~Grimm, S.~Fang, Y.~Deng, and T.M. Garoni.
\newblock Random-length random walks and finite-size scaling in high
  dimensions.
\newblock {\em Phys. Rev. Lett.}, {\bf 121}:185701, (2018).

\end{thebibliography}

\end{document}